\documentclass[12pt]{amsart}
\usepackage[T1]{fontenc}
\usepackage[a4paper,top=4cm, bottom=5cm, left=2.75cm, right=2.75cm]{geometry}
\usepackage[utf8]{inputenc}
\usepackage[dvipsnames]{xcolor}
\usepackage{amsthm,amsmath,bm}  
\usepackage{latexsym,amssymb}
\usepackage{stmaryrd}
\usepackage[shortlabels]{enumitem}
\usepackage{booktabs}
\usepackage[hidelinks]{hyperref}
\usepackage{orcidlink}

\usepackage{multicol}
\usepackage{fancyvrb}

\usepackage{algorithm}
\usepackage{algpseudocode}

\usepackage{pgf,tikz,pgfplots}
\usepackage{graphicx}
\usepackage{caption}
\usepackage{subcaption}
\usepackage{booktabs}
\pgfplotsset{compat=1.18}

\newcommand{\N}{\mathbb{N}}
\newcommand{\Z}{\mathbb{Z}}
\newcommand{\Q}{\mathbb{Q}}

\newcommand{\mB}{\mathcal{B}}
\newcommand{\mG}{\mathcal{G}}
\newcommand{\mH}{\mathcal{H}}
\newcommand{\mA}{\mathcal{A}}
\newcommand{\mS}{\mathcal{S}}

\newcommand{\mE}{\mathcal{E}}
\newcommand{\mF}{\mathcal{F}}
\newcommand{\mX}{\mathcal{X}}
\newcommand{\Pn}[1]{\mathbb{P}^{\, #1}}

\newcommand{\bfs}{\mathbf{s}}
\newcommand{\bx}{\mathbf{x}}
\newcommand{\bft}{\mathbf{t}}
\newcommand{\bfe}{\mathbf{e}}
\newcommand{\bfa}{\mathbf{a}}
\newcommand{\bfb}{\mathbf{b}}
\newcommand{\bfc}{\mathbf{c}}

\newcommand{\bfh}{\mathbf{h}}
\newcommand{\bfeps}{\boldsymbol\epsilon}

\newcommand{\hf}[1]{{\rm HF}_{#1}}

\newcommand{\hs}[1]{{\rm HS}_{#1}}
\newcommand{\reg}[1]{{\rm reg}(#1)}
\newcommand{\depth}[1]{{\rm depth}(#1)}

\newcommand{\degs}[1]{| #1 |_\mS}
\newcommand{\degw}{\deg_\omega}

\newcommand{\AP}{{\rm AP}}
\newcommand{\sg}{\mathcal{S}}
\newcommand{\aps}{\AP_\sg}

\renewcommand{\k}{\Bbbk}
\newcommand{\kx}{\Bbbk[x_1,\ldots,x_n]}
\newcommand{\kxd}{\Bbbk[x_{n-d+1},\ldots,x_n]}
\newcommand{\ks}{\Bbbk[\mS]}
\newcommand{\kt}{\Bbbk[t_1,\ldots,t_d]}
\newcommand{\ini}[1]{{\rm in}(#1)}
\newcommand{\id}[1]{\langle #1 \rangle}

\DeclareMathOperator{\lcm}{lcm}

\def\whom{$\omega$-homogeneous}
\def\iff{if and only if}

\theoremstyle{plain}
\newtheorem{theorem}{Theorem}[section]
\newtheorem{lemma}[theorem]{Lemma}
\newtheorem{corollary}[theorem]{Corollary}
\newtheorem{proposition}[theorem]{Proposition}

\theoremstyle{definition}
\newtheorem{definition}[theorem]{Definition}
\newtheorem{example}[theorem]{Example}
\newtheorem{remark}[theorem]{Remark}

\title{Computational aspects of the short resolution}

\author[I. García-Marco]{Ignacio García-Marco\,\orcidlink{0000-0003-4993-7577}}
\address{Instituto de Matem\'aticas y Aplicaciones (IMAULL), Secci\'on de Matem\'aticas, Fa\-cul\-tad de Ciencias, Universidad de La Laguna, 38200, La Laguna, Spain}
\email{iggarcia@ull.edu.es}

\author[P. Gimenez]{Philippe Gimenez\,\orcidlink{0000-0002-5436-9837}}
\address{Instituto de Investigaci\'on en Matem\'aticas de la Universidad de Valladolid (IMUVA), Universidad de Valladolid, 47011 Valladolid, Spain.}
\email{pgimenez@uva.es}

\author[M. Gonz\'alez-S\'anchez]{Mario Gonz\'alez-S\'anchez\,\orcidlink{0000-0002-6458-7547}}
\address{Instituto de Investigaci\'on en Matem\'aticas de la Universidad de Valladolid (IMUVA), Universidad de Valladolid, 47011 Valladolid, Spain.}
\email{mario.gonzalez.sanchez@uva.es}
 
\thanks{
This work was supported in part by the grant PID2022-137283NB-C22 funded by MICIU/AEI/ 10.13039/501100011033 and by ERDF/EU.
The third author thanks financial support from European Social Fund, {\it Programa Operativo de Castilla y Le\'on}, and {\it Consejer\'ia de Educaci\'on de la Junta de Castilla y Le\'on}.
}

\subjclass[2020]{14Q15, 13D02, 14Q10, 13F65, 20M50}

\keywords{Graded free resolution, Betti numbers, Noether normalization, toric ring, simplicial semigroup}

\begin{document}

\begin{abstract}
Let $R:= \Bbbk[x_1,\ldots,x_{n}]$ be a polynomial ring over a field $\Bbbk$, $I \subset R$ be a homogeneous ideal with respect to a weight vector $\omega = (\omega_1,\ldots,\omega_n) \in (\mathbb{Z}^+)^n$, and denote by $d$ the Krull dimension of $R/I$. In this paper we study graded free resolutions of $R/I$ as $A$-module whenever $A :=\Bbbk[x_{n-d+1},\ldots,x_n]$ is a Noether normalization of $R/I$. We exhibit a Schreyer-like method to compute a (non-necessarily minimal) graded free resolution of $R/I$ as $A$-module. When $R/I$ is a $3$-dimensional simplicial toric ring, we describe how to prune the previous resolution to obtain a minimal one. We finally provide an example of a $6$-dimensional simplicial toric ring whose Betti numbers, both as $R$-module and as $A$-module, depend on the characteristic of $\Bbbk$.
\end{abstract}

\maketitle

\section*{Introduction}
Let $\k$ be an arbitrary field, $R = \kx$ a polynomial ring over $\k$, and $I \subset R$ a $\omega$-homogeneous ideal for some weight vector $\omega = (\omega_1,\ldots,\omega_n) \in (\Z^+)^n$, i.e., $I$ is homogeneous for the grading on $R$ induced by $\degw (x_i) = \omega_i$. 
We denote by $d := \dim(R/I)$ the Krull dimension of $R/I$ and assume that $A := \kxd$ is a {\it Noether normalization} of $R/I$, that is, $A \hookrightarrow R/I$ is an integral ring extension. 
When this occurs, we will say that the variables are in {\it Noether position}.
In this setting, $R/I$ is a finitely generated graded $A$-module, so it has a finite minimal graded free resolution as $A$-module. This resolution has been referred to in the literature as the {\it short resolution} \cite{Ojeda2017,Pison2003} or {\it Noether resolution} \cite{BGGM2017} of $R/I$. 
We denote it by 
\begin{equation} \label{eq:shortresol}
 \mathcal{F}: 0 \rightarrow \oplus_{v\in \mB_p} A(-s_{p,v}) \xrightarrow{\psi_p} \dots \xrightarrow{\psi_1} \oplus_{v\in \mB_0} A(-s_{0,v}) \xrightarrow{\psi_0} R/I \rightarrow 0 \, ,  
\end{equation}
where its length $p={\rm pd}_A(R/I)$ is the projective dimension of $R/I$ as $A$-module, and for all $i\in \{0,\ldots,p\}$, the sets $\mB_i \subset R$ are finite, and $s_{i,v}$ are nonnegative integers.\newline

The relation between the lengths of the short resolution of $R/I$ and of its usual minimal graded free resolution as $R$-module is given by ${\rm pd}_R (R/I) = {\rm pd}_A (R/I) + n-d$. This follows from the Auslander-Buchsbaum formula and the fact that ${\rm depth}_A(R/I) = {\rm depth}_R(R/I)$; see, e.g. \cite[Ex.~1.2.26(b)]{BrunsHerzog1998}. Hence, the short resolution is shorter than the usual minimal graded free resolution, and it contains valuable combinatorial, algebraic and geometric information about $R/I$. For example, since (\ref{eq:shortresol}) is a graded free resolution of $R/I$, one gets that the (weighted) Hilbert series of $R/I$ can be expressed as:
\[\hs{R/I}(t) = \dfrac{\sum_{i=0}^{p} \sum_{v \in \mB_i} (-1)^i \, t^{s_{i,v}}}{(1-t)^d}, \]
and its numerator, $h(t) = \sum_{i=0}^{p} \sum_{v \in \mB_i} (-1)^i \, t^{s_{i,v}}$, satisfies that $h(1) = \sum_{i = 0}^p (-1)^i |\mB_i|$ is $e(R/I)$, the degree (or multiplicity) of $R/I.$ 
Moreover, when $I$ is homogeneous with respect to the standard grading, as a consequence of the Independence Theorem for local cohomology (see, e.g. \cite[Sect.~1]{Symonds2011}), the Castelnuovo-Mumford regularity of $R/I$, $\reg{R/I}$, can be computed using the short resolution:
\[\reg{R/I} = \max \{ s_{i,v}-i \mid 0 \leq i \leq p , v \in \mB_i\}\,.\]

In \cite{BGGM2017}, the authors describe how to compute short resolutions in some cases. The first step of the short resolution is given by \cite[Prop.~1]{BGGM2017} that we recall in Proposition~\ref{prop:first_step}. This result provides the whole short resolution when $R/I$ is Cohen-Macaulay. 
If $R/I$ is not Cohen-Macaulay, the resolution has at least one more step. When $\dim(R/I) = 1$ and $\depth{R/I} = 0$, the second (and last) step of the short resolution is given in \cite[Prop.~3]{BGGM2017}. Moreover, when $\dim(R/I) = 2$ and $x_n$ is not a zero divisor on $R/I$, the whole short resolution is given in \cite[Prop.~4]{BGGM2017}. 
In the first section, we study the short resolution in any dimension, and we also drop the assumption that $x_n$ is a nonzero divisor on $R/I$. We will only assume that $I$ is homogeneous for some grading $\omega \in (\Z^+)^n$, and that $A$ is a Noether normalization. Note that this last assumption is not restrictive if $I$ is homogeneous for the standard grading and $\k$ is infinite since linear changes of coordinates preserve homogeneity for the standard grading, and
$A$ is a Noether normalization of $R/I$ after a generic linear change of coordinates; see 
\cite[Lem.~4.1]{BGi01} for a Noether position test, and \cite[App.~A]{Bermejo2006sat} for smaller changes of coordinates.
\newline

Our main results in the first section are Proposition~\ref{prop:ker_psi0} and Theorem~\ref{them:GBschreyer}. In Proposition~\ref{prop:ker_psi0}, 
using the monomial generators of $R/I$ as $A$-module given in \cite[Prop.~1]{BGGM2017}, we describe a generating set (that may not be minimal) of its module of syzygies, a submodule of a free $A$-module.
This presentation of the $A$-module $R/I$ by generators and relations allows to obtain its 
minimal graded free resolution by means of standard $A$-module computations. This gives the first way of contructing the short resolution of $R/I$ (Algorithm~\ref{alg:short_res}).
Another way to obtain the short resolution is as follows.
In Theorem~\ref{them:GBschreyer}, we prove that the generating set given in Proposition~\ref{prop:ker_psi0} is, indeed, the reduced Gr\"obner basis of the syzygy submodule for a Schreyer-like monomial order, and hence we can build a graded free resolution (that does not need to be minimal) of $R/I$ as $A$-module by an iterative application of Schreyer's theorem. 
\newline

A case in which our results apply nicely is that of toric rings. For a finite set of nonzero vectors $\mA = \{\bfa_1,\ldots,\bfa_n\} \subset \N^d$ with $\bfa_i = (a_{i1},\ldots,a_{id}) \in \N^d$, let $\mS_\mA \subset \N^d$ be the affine monoid spanned by $\mA$,
\[\mS_\mA = \{\lambda_1 \bfa_1 + \dots + \lambda_n \bfa_n \mid \lambda_1,\ldots,\lambda_n \in \N\} \, .\]
The toric ideal determined by $\mA$, $I_\mA$, is the kernel of the homomorphism of $\k$-algebras $\varphi_\mA: \kx \rightarrow \kt$ induced by $x_i \mapsto \bft^{\bfa_i} = t_1^{a_{i1}} t_2^{a_{i2}} \dots t_d^{a_{id}}$. The ideal $I_\mA$ is prime and, if one sets the $\mS_\mA$-degree of a monomial $\bx^{\alpha} \in \kx$ as $|\bx^\alpha|_{\mS_\mA} :=
\alpha_1 \bfa_1 + \cdots + \alpha_n \bfa_n \in \mS_\mA$, it is $\mS_\mA$-homogeneous, i.e., $I_\mA$ is homogeneous for the grading induced by $|\bx^\alpha|_{\mS_\mA}$.
Indeed, $I_{\mA}$ is the binomial ideal given by:
\[ I_{\mA} = \langle \, \bx^\alpha - \bx^\beta \ \vert \  |\bx^\alpha|_{\mS_\mA} = |\bx^\beta|_{\mS_\mA} \, \rangle. \]
The toric ring $\k[x_1,\ldots,x_n] / I_\mA$ and the semigroup algebra $\k[\mS_\mA] := \k[t^{\bf s} \, \vert \, {\bf s} \in \mS_\mA]$ are isomorphic as $\mS_\mA$-graded algebras. We say that the semigroup $\mS_\mA$ is simplicial when the rational cone spanned by $\mA$, i.e., ${\rm Cone}(\mA) := \{\sum_{i = 1}^n \lambda_i \bfa_i \, \vert \, \lambda_i \in \mathbb R_{\geq 0}\} \subset \mathbb R^d$, has dimension $d$ and is minimally generated by $d$ rays. 
One has that $A$ is a Noether nomalization of $R/I_\mA$ if and only if ${\rm Cone}(\mA)$ is spanned by $\bfa_{n-d+1},\ldots,\bfa_n$; when this happens, $\mS_\mA$ is a simplicial semigroup and we say that $R/I_\mA$ is a simplicial toric ring. 
Short resolutions of simplicial toric rings were first studied in \cite{Pison2003} and later in \cite{Ojeda2017} under the name of Pison's resolution. Notably, one can find in these papers an algorithm involving Gr\"obner basis for computing a presentation of $\k[\mS_\mA]$, and Hochster-like formulas for the Betti numbers of these resolutions in terms of some simplicial homology. 
In Sections~\ref{sec:homology} and~\ref{sec:pruning}, 
we study three dimensional simplicial toric rings and their short resolution. In Section~\ref{sec:homology}, we describe their short resolution and their Hilbert series and function in terms of the combinatorics of the associated semigroup translating some results of \cite{Pison2003} and \cite{Ojeda2017}. When $I_\mA$ is standard-graded homogeneous, we provide formulas for the Castelnuovo-Mumford regularity of the toric ring.
In Section~\ref{sec:pruning}, we devise an algorithm to compute the short resolution for 3-dimensional simplicial toric rings. This algorithm first constructs a non-minimal graded free resolution as $A$-module following Section~\ref{sec:groebner} (Algorithm~\ref{alg:dim3}), and then minimalizes/prunes it to obtain the short one by applying Theorems~\ref{thm:prunB1} and~\ref{thm:prunB2} (Algorithm~\ref{alg:pruning}). The whole algorithm involves the computation of the reduced Gr\"obner bases of $I_\mA$ and $I_\mA + \langle x_{n-2} \rangle$, and the division of some monomials by those bases.   \newline

Several authors have addressed the problem of determining which algebraic invariants of a semigroup algebra depend on the field $\k$. For general toric ideals, it is known that binomial generating sets and reduced Gr\"obner bases of $I_{\mA}$, and also several Betti numbers of $\k[\mS_\mA]$, are independent of $\k$; see, e.g., \cite{Sturm} and \cite[Thm.~1.3]{BrunsHerzog1997}. Nevertheless, the Gorenstein, Cohen-Macaulay and Buchsbaum properties of $\k[\mS_\mA]$ depend on the characteristic of $\k$ in general; see \cite{Hoa1991}, \cite{TH1986} and \cite{Hoa1988}, respectively. This situation changes in the context of simplicial semigroup rings since the Gorenstein, Cohen-Macaulay, and Buchsbaum properties can be completely described in terms of the combinatorics of $\mS_\mA$, and hence do not depend on $\k$; see \cite{GSW1976}, \cite{Stanley1978} and \cite{GarciaSanchezRosales02}, respectively. These facts could lead to conjecture that the Betti numbers of simplicial semigroup algebras do not depend on $\k$ but this does not hold. In the last section, we provide an example of a simplicial semigroup whose algebra has different projective dimensions, both as $A$-module and as $R$-module, depending on the characteristic of $\k$.
To our knowledge, this is the first example in which this phenomenon is observed. \newline

The computations in the examples given in this paper have been performed using SageMath \cite{Sage}. Our algorithms have been implemented in SageMath and are available in the GitHub repository \cite{github_shortres}.

\section{Contruction of the short resolution using Gröbner bases} \label{sec:groebner}

Let $\omega = (\omega_1,\ldots,\omega_n) \in (\Z^+)^n$ be a weight vector,  $\Bbbk$ an arbitrary field and $R = \kx$. Consider $I \subset R$ a \whom{} ideal, i.e., a homogeneous ideal with respect to the grading induced by $\deg_\omega(x_i) = \omega_i$ for all $i \in \{1,\ldots,n\}$. Take $d := \dim(R/I)$ and assume that $A = \kxd$ is a Noether normalization of $R/I$. In this section we study the {\it short resolution} of $R/I$, i.e., the minimal graded free resolution of $R/I$ as $A$-module:
\begin{equation}
\mathcal{F}: 0 \rightarrow \oplus_{v\in \mB_p} A(-s_{p,v}) \xrightarrow{\psi_p} \dots \xrightarrow{\psi_1} \oplus_{v\in \mB_0} A(-s_{0,v}) \xrightarrow{\psi_0} R/I \rightarrow 0 \, ,
\label{eq:NoetherRes}
\end{equation}
where $p = {\rm pd}_A(R/I),$ and for all $i\in \{0,\ldots,p\}$, $\mB_i \subset R$ is a finite set and $s_{i,v}$ are nonnegative integers. In our description, the sets $\mB_i$ will consist of monomials and $s_{i,v} = \deg_\omega(v)$ will be the $\omega$-degree of the monomial $v\in \mB_i$. Note that the sets $\mB_i$ might not be unique, but their degrees are.
\newline

Consider the $\omega$-graded reverse lexicographic order $>_\omega$ in $R$, i.e., the monomial order defined as follows: $\bx^\alpha >_\omega \bx^\beta$ \iff{} \begin{itemize}
    \item $\degw(\bx^\alpha) > \degw(\bx^\beta)$, or
    \item $\degw(\bx^\alpha) = \degw(\bx^\beta)$ and the last nonzero entry of $\alpha-\beta \in \Z^n$ is negative.
\end{itemize}
For every polynomial $f\in R$, let $\ini{f}$ denote the initial term of $f$ with respect to $>_\omega$ (we include the coefficient in the initial term). Given an ideal $J\subset R$, $\ini{J}$ denotes the initial ideal of $J$ with respect to $>_\omega$, and $\mG$ the reduced Gröbner basis of $I$ with respect to $>_\omega$. Since $I$ is $\omega$-homogeneous, $\mathcal G$ consists of $\omega$-homogeneous polynomials. \newline

With these notations, the first step of the short resolution of $R/I$ is given by the following result:

\begin{proposition}[{\cite[Prop.~1]{BGGM2017}}] \label{prop:first_step}
Let $\mB_0 \subset R$ be the set of monomials that do not belong to $\ini{I} + \id{x_{n-d+1},\ldots,x_n}$. Then, \[ \{u+I \mid u \in \mB_0\}\]  is a minimal set of generators of $R/I$ as $A$-module.
The $\omega$-graded $A$-module homomorphism $\psi_0: \oplus_{v\in \mB_0} A \left( -\deg_\omega(v) \right) \rightarrow R/I$ is defined by $\psi_0(\bfeps_u) = u+I$, where $\{\bfeps_u \mid u \in \mB_0\}$ denotes the canonical basis of $\oplus_{v\in \mB_0} A \left( -\deg_\omega(v) \right)$, and hence the shifts at the first step of the short resolution \eqref{eq:NoetherRes} are the $\omega$-degrees of the elements $u \in \mB_0$.
\end{proposition}

This result provides the whole short resolution when $R/I$ is a free $A$-module, i.e., when the projective dimension of $R/I$ as $A$-module is $0$, which is equivalent to $R/I$ being Cohen-Macaulay. In Gr\"obner basis terms, this is also equivalent to the fact that variables $x_{n-d+1},\ldots,x_n$ do not divide any minimal generator of $\ini I$; see, e.g., \cite[Thm.~2.1]{BGi01} or \cite[Prop.~2]{BGGM2017}. 
\newline

When $R/I$ is not free, the resolution has at least one more step. In this case, we will describe the relations between the generators of $R/I$ given in Proposition~\ref{prop:first_step}, i.e., provide a finite set of generators $\mH$ of $\ker(\psi_0)$ that may not be minimal. This gives a presentation of $R/I$ as $A$-module: $R/I$ is isomorphic to the quotient of a free $A$-module by the submodule generated by $\mH$, and the short resolution of $R/I$ can then be obtained by standard $A$-module computations. \newline

Let $\chi: R\rightarrow R$ be the evaluation morphism defined by $\chi(x_i) = x_i$ for $i \in \{1,\ldots,n-d\}$ and $\chi(x_j) = 1$ for $j \in \{n-d+1, \ldots,n\}$, and set $J := \chi\left( \ini{I} \right) . R$, the extension of the ideal $\ini{I}$ by the ring homomorphism $\chi$. Now, for every monomial $u \in \mB_0 \cap J$, consider the ideal $I_u$ defined by 
\[I_u := \left( \ini{I} : u \right) \cap \kxd  \, .\] 
Since $I_u$ is a monomial ideal, it has a unique minimal monomial generating set denoted by $G(I_u)$, and let $\mB_1'$ be the following set of monomials: 
\begin{equation}\label{eq:Bprime1}
\mB_1' = \{u \cdot M \mid u \in \mB_0 \cap J , M \in G(I_u) \}. 
\end{equation} 
Each monomial $\bx^\alpha \in \mB_1'$ can be written uniquely as $\bx^\alpha = u \cdot M_\alpha$, where $u = \chi(\bx^\alpha) \in \mB_0 \cap J$ and $M_\alpha \in G(I_u)$.
Let $r_\alpha$ be the remainder of the division of $\bx^\alpha$ by $\mG$, the reduced Gröbner basis of $I$ with respect to $>_\omega$. Since every monomial  in the expression of $r_\alpha$ does not belong to $\ini{I}$, one can 
uniquely write  $r_\alpha = \sum_{v\in \mB_0} f_{\alpha,v} v$ with $f_{\alpha,v} \in A$.
Using these notations, for all $\bx^\alpha \in \mB_1'$ set
\begin{equation}\label{eq:h}
\bfh_\alpha := M_\alpha \cdot \bfeps_u - \sum_{v \in \mB_0} f_{\alpha,v} \cdot \bfeps_v \in \oplus_{v \in \mB_0} A\left( -\deg_\omega(v) \right) ,
\end{equation}
where $\{\bfeps_v \mid v\in \mB_0\}$ denotes the canonical basis of $\oplus_{v \in \mB_0} A\left( -\deg_\omega(v) \right)$.
\newline

Since $\psi_0 (\bfh_\alpha) = \left( \bx^\alpha-r_\alpha \right) +I = 0$ for all $\bx^\alpha \in \mB_1'$, one has that $\langle \bfh_\alpha \mid \bx^\alpha \in \mB_1' \rangle \subset \ker(\psi_0)$. The next result shows that this inclusion is indeed an equality.

\begin{proposition} \label{prop:ker_psi0}
The kernel of the $A$-module homomorphism $\psi_0$ is \[\ker(\psi_0) = \langle \bfh_\alpha \mid \bx^\alpha \in \mB_1' \rangle \, .\]
\end{proposition}

\begin{proof}
Consider $\bar{\psi}_0: \oplus_{v\in \mB_0} A\left( -\deg_w(v) \right) \rightarrow R$ the $A$-module homomorphism defined by $\bar{\psi}_0(\bfeps_v) = v$, for all $v\in \mB_0$. Take $g \in \ker(\psi_0)$, and let us prove that $g \in \langle \bfh_\alpha \mid \bx^\alpha \in \mB_1' \rangle$. We write $g = \sum_{v\in \mB_0} g_v \bfeps_v$ with $g_v \in A$ for all $v \in \mB_0$. Since $g \in \ker(\psi_0)$, then $g' = \bar{\psi}_0(g) = \sum_{v \in \mB_0} g_v \cdot v \in I$ and its initial term is $\ini{g'} = c \cdot w \cdot M_\gamma$ for some $c\in \Bbbk \setminus \{0\}$, $w\in \mB_0$ and a monomial $M_\gamma \in A$. In fact, $w \in \mB_0 \cap J$ and $M_\gamma \in I_w = \left( \ini{I} : w \right) \cap A$. Hence, there exists $\bx^\alpha = M_\alpha w \in \mB_1'$ such that $M_\alpha$ divides $M_\gamma$. Let us consider $g_1 = g-c \frac{M_\gamma}{M_\alpha} \bfh_\alpha \in \ker(\psi_0)$. If $g_1=0$, then $g\in \langle \bfh_\alpha \mid \bx^\alpha \in \mB_1'\rangle$. Otherwise, one has that $0 \neq \ini{\bar{\psi}_0(g_1)} < \ini{\bar{\psi}_0 (g)}$ and we iterate this process. The result then follows by induction because $>_\omega$ is a well order.
\end{proof}

Proposition~\ref{prop:ker_psi0} provides a system of generators of $\ker(\psi_0)$. As a consequence, we get the next step of a non-necessarily minimal graded free resolution of $R/I$ as $A$-module.

\begin{corollary}\label{cor:segundopaso}
Consider the morphism of $A$-modules 
\[\begin{split}
\psi_1': \oplus_{\bx^\alpha \in \mB_1'} A \left( -\degw(\bx^\alpha) \right) &\rightarrow \oplus_{v\in \mB_0} A \left( -\degw(v) \right) \\
\bfeps_\alpha &\mapsto \bfh_\alpha
\end{split}\]
where $\{\bfeps_\alpha \mid \bx^\alpha \in \mB_1'\}$ is the canonical basis of $\oplus_{\bx^\alpha \in \mB_1'} A \left( -\degw(\bx^\alpha) \right)$. Then, ${\rm Im}(\psi_1') = {\rm Ker}(\psi_0).$
\end{corollary}

Since $R/I$ and $\oplus_{v\in \mB_0} A \left( -\degw(v) \right) / \ker(\psi_0)$ are isomorphic as graded $A$-modules, their minimal graded free resolutions coincide up to isomorphism. Thus, one can compute the short resolution by applying standard $A$-module computations to the submodule $\ker(\psi_0) = \langle \bfh_\alpha \mid \bx^\alpha \in \mB_1' \rangle \subset \oplus_{v\in \mB_0} A(-\deg_\omega(v))$. The whole process to obtain the short resolution of $R/I$ is shown in Algorithm~\ref{alg:short_res}. It has been implemented in the function {\tt shortRes} of \cite{github_shortres}. \newline

\begin{algorithm} 
\caption{Computation of the short resolution.} \label{alg:short_res}
\begin{flushleft}
    \textbf{Input:} $I \subset R$ a weighted homogeneous ideal with variables in Noether position \\
    \textbf{Output:} Short resolution of $R/I$
\end{flushleft}
\begin{algorithmic}[1]
\State $\mG$  $\gets$ reduced Gröbner basis of $I$ for $>_\omega$.
\State $\mB_0$ $\gets$ $\k$-basis of $\ini{I}+\id{x_{n-d+1},\ldots,x_n}$ for  $>_\omega$.
\State $J \gets \chi \left( \ini{I} \right) . R$, where $\chi: R \rightarrow R$ is defined by $\chi(x_i) = x_i$ for $i \in \{1,\ldots,n-d\}$, and $\chi(x_j) = 1$ for $j \in \{n-d+1,\ldots,n\}$.
\State $I_u \gets \left( \ini{I} : u \right) \cap A$, $\forall u \in \mB_0 \cap J$.
\State $G(I_u) \gets$ minimal monomial generating set of $I_u$, $\forall u \in \mB_0 \cap J$.
\State $\mB_1' \gets \{ u \cdot M \mid u \in \mB_0 \cap J, M \in G(I_u)\}$.
\State $r_\alpha \gets $ remainder of $\bx^\alpha$ by $\mG$, $\forall \bx^\alpha \in \mB_1'$.
\State For all $\bx^\alpha \in \mB_1'$, write $\bx^\alpha = M_\alpha u$ and $r_\alpha = \sum_{v\in \mB_0} f_{\alpha,v} v$.
\State $\bfh_\alpha \gets M_\alpha \bfeps_u - \sum_{v \in \mB_0} f_{\alpha,v} \bfeps_v$, $\forall \bx^\alpha \in \mB_1'$.
\State $\ker(\psi_0) \gets \langle \bfh_\alpha \mid \bx^\alpha \in \mB_1' \rangle$.
\State Compute the m.g.f.r. of $\ker(\psi_0)$.
\end{algorithmic}
\end{algorithm}

\begin{example}\label{ex:dim2_zerodiv_alg1}
Set $R = \Q[x_1,x_2,x_3,x_4,x_5]$,  
let $>$ be the degree reverse lexicographic order in $R$, and consider $I_{\mathcal{C}} \subset R$, the defining ideal of the projective monomial curve parametrized by the sequence $0<1<2<6<7$, i.e.,
\[ I_{\mathcal{C}} = \id{x_1-t_1t_2^6,x_2-t_1^2 t_2^5,x_3-t_1^6 t_2,x_4-t_1^7,x_5-t_2^7} \cap \Q[x_1,x_2,x_3,x_4,x_5] \, . \]
Let $L$ be the zero-dimensional ideal $L = \id{x_1^2-x_2x_3,x_2^3-x_4x_5^2,x_1x_2,x_3^2,x_4^2-x_5^2,x_2x_5,x_5^4}$, and consider the ideal $I = I_{\mathcal{C}} \cap L$. One has that $I$ is homogeneous, $\dim(R/I)=\dim(R/I_{\mathcal{C}})=2$, and variables are in Noether position, i.e., $A = \Q[x_4,x_5]$ is a Noether normalization of $R/I$. Moreover, $x_5$ is a zero divisor on $R/I$ because $f = x_2x_3^2-x_4^2x_5 \notin I$ while $fx_5 \in I$, so we are not under the hypotheses of \cite[Prop.~4]{BGGM2017}.
By Proposition~\ref{prop:first_step}, a minimal system of generators of $R/I$ as $A$-module is $\{u+I\mid u \in \mB_0\}$ for
\[\begin{split}
\mB_0 = \{ &{} u_{1}=x_{3}^{4},u_{2}=x_{3}^{3},u_{3}=x_{2} x_{3}^{2},u_{4}=x_{1} x_{3}^{2},u_{5}=x_{3}^{2},u_{6}=x_{2} x_{3},u_{7}=x_{1} x_{3},u_{8}=x_{3},\\
&{} u_{9}=x_{2}^{3}, u_{10}=x_{1} x_{2}^{2},u_{11}=x_{2}^{2},u_{12}=x_{1} x_{2},u_{13}=x_{2},u_{14}=x_{1}^{2},u_{15}=x_{1},u_{16}=1 \} \, .
\end{split}\]
If $\chi:R\rightarrow R$ is the ring homomorphism defined by $\chi(x_1)=x_1$, $\chi(x_2)=x_2$, $\chi(x_3)=x_3$, and $\chi(x_4)=\chi(x_5)=1$, then
$J = \chi\left( \ini{I} \right) . R= \id{x_3^5,x_1^2,x_1x_2,x_2^2,x_1x_3,x_2x_3}$, and hence
$\mB_0 \cap J = \{ u_3, u_4, u_6, u_7, u_9,u_{10}, u_{11}, u_{12}, u_{14} \}$, and
$I_{u_3} = I_{u_4} =\id{x_4,x_5}$, $I_{u_6} = \id{x_4^2,x_5^2}$, $I_{u_7} = \id{x_5^4,x_4x_5^2,x_4^2}$, $I_{u_9} = \id{x_5^2,x_4}$, $I_{u_{10}} = \id{x_4}$, $I_{u_{11}} = \id{x_4^2}$, $I_{u_{12}} = \id{x_4^3}$, and $I_{u_{14}} = \id{x_4^2,x_5}$. 
Thus, the set $\mB_1'$ defined in \eqref{eq:Bprime1} is
\[\begin{split}
\mB_1' = \{ &{} x_{2} x_{3}^{2} x_{4}, \, x_{2} x_{3}^{2} x_{5}, \, x_{1} x_{3}^{2} x_{4}, \, x_{1} x_{3}^{2} x_{5}, \, x_{2} x_{3} x_{4}^{2}, \, x_{2} x_{3} x_{5}^{2}, \, x_{1} x_{3} x_{5}^{4}, \,  x_{1} x_{3} x_{4} x_{5}^{2}, \, x_{1} x_{3} x_{4}^{2}, \\
&{} x_{2}^{3} x_{5}^{2}, \, x_{2}^{3} x_{4}, \, x_{1} x_{2}^{2} x_{4}, \, x_{2}^{2} x_{4}^{2}, \, x_{1} x_{2} x_{4}^{3}, \, x_{1}^{2} x_{4}^{2}, \, x_{1}^{2} x_{5} \}. 
\end{split}\]
Take the first element in $\mB_1'$, $\bx^\alpha=x_{2} x_{3}^{2} x_{4}=x_4u_3$, and compute the remainder $r_\alpha$ of its division by the reduced Gröbner basis of $I$ with respect to $>$: $r_\alpha=x_4^3x_5=x_4^3x_5u_{16}$. The corresponding element as in \eqref{eq:h} is $\bfh_\alpha = x_4 \bfeps_3 - x_4^3 x_5 \bfeps_{16}$
where $\{\bfeps_1,\ldots,\bfeps_{16}\}$ is the canonical basis of $\oplus_{i=1}^{16} A \left( -\deg(u_i) \right)$.
Doing the same for each monomial in $\mB_1'$, one gets 16 elements that generate the submodule $\ker(\psi_0)$ of the free module $\oplus_{i=1}^{16} A \left( -\deg(u_i) \right)$, and computing the minimal graded free resolution of this submodule, one gets the short resolution of $R/I$. 
Using the function {\tt shortRes} of \cite{github_shortres}, one gets directly the Betti table of the short resolution: 
\begin{center}
\begin{verbatim}
           0     1     2
------------------------
    0:     1     -     -
    1:     3     -     -
    2:     6     1     -
    3:     5    11     2
    4:     1     2     2
    5:     -     -     1
------------------------
total:    16    14     5
\end{verbatim}
\end{center}
Observe that in this example the set of 16 generators of $\ker(\psi_0)$ given by Proposition~\ref{prop:ker_psi0} is not minimal since the Betti table shows that $\ker(\psi_0)$ is minimally generated by 14 elements. We will come back to this example later in Example~\ref{ex:dim2_zerodiv}.
\end{example}

Interestingly, the system of generators provided in Proposition~\ref{prop:ker_psi0} is, in fact, a Gröbner basis for a monomial order in $\oplus_{v\in \mB_0} A(-\deg_\omega(v))$ that we now introduce. This can be used to provide another method for computing a graded free resolution of $R/I$ as $A$-module. 
For general results concerning Gröbner bases for submodules of the free $A$-module $A^m$, we refer the reader to \cite[Chap.~5]{CLO2} and \cite[Chap.~15]{Eis95}.

\begin{definition}\label{def:SLorder}
Consider the monomial order $>_{ \rm SL}$ in $\oplus_{v\in \mB_0} A(-\deg_\omega(v))$ defined as follows: for all $M,M' \in A$ monomials and $u,v \in \mB_0$, \[M \bfeps_u >_{\rm SL} M' \bfeps_v \Longleftrightarrow u \cdot M >_\omega v \cdot M' \, .\]
We call this monomial order the {\it Schreyer-like order} in $\oplus_{v \in \mB_0} A \left( -\degw(v) \right)$.
\end{definition}

If $\bar{\psi}_0$ is the homomorphism of $A$-modules introduced in the proof of Proposition~\ref{prop:ker_psi0}, $\bar{\psi}_0: \oplus_{v\in \mB_0} A\left( -\deg_\omega(v) \right) \rightarrow R$, $\bfeps_v\mapsto v$,
it is injective and maps monomials to monomials, and \[M \bfeps_u >_{\rm SL} M' \bfeps_v \Longleftrightarrow \bar{\psi}_0(M) >_\omega \bar{\psi}_0(M') \, .\]
This equivalent description of $>_{\rm SL}$ proves that it is a monomial order and justifies its name.

\begin{remark} \label{rem:initial}
For each $\bx^\alpha \in \mB_1'$, the initial term of $\bfh_\alpha = M_\alpha \cdot \bfeps_u - \sum_{v \in \mB_0} f_{\alpha,v} \cdot \bfeps_v$ for the Schreyer-like monomial order $>_{\rm SL}$ is $\ini{\bfh_\alpha} = M_\alpha \cdot \bfeps_u$.
\end{remark}

\begin{theorem}\label{them:GBschreyer}
The set $\mH = \{\bfh_\alpha \mid \bx^\alpha \in \mB_1'\}$ is the reduced Gröbner basis of $\ker(\psi_0)$ for the Schreyer-like order $>_{\rm SL}$.
\end{theorem}

\begin{proof}
By Proposition~\ref{prop:ker_psi0}, $\ker(\psi_0) = \langle \mH \rangle$. By Buchberger's criterion, $\mH$ is a Gr\"obner basis if and only if,
for all $\bfh_\alpha,\bfh_\beta \in \mH$, the $S$-polynomial $S(\bfh_\alpha,\bfh_\beta)$ reduces to zero modulo $\mH$. One has that $S(\bfh_\alpha,\bfh_\beta)=0$ whenever $\ini{\bfh_\alpha}$ and $\ini{\bfh_\beta}$ are 
multiples of different elements in the canonical basis $\{\bfeps_v \mid v\in \mB_0\}$.
Let $\bfh_\alpha,\bfh_\beta$ be two elements in $\mH$ whose initial terms are 
multiples of the same element in the canonical basis.
By Remark~\ref{rem:initial}, there exist monomials $u \in \mB_0$ and $M_\alpha,M_\beta \in A$, such that $\ini{\bfh_\alpha} = M_\alpha \bfeps_u$ and $\ini{\bfh_\beta} = M_\beta \bfeps_u$. Set $h_\alpha' = u M_\alpha - r_\alpha$ and $h_\beta' = u M_\beta -r_\beta$, where $r_\alpha$ and $r_\beta$ are the remainder of the division of $\bx^\alpha = u M_\alpha $ and $\bx^\beta = u M_\beta $ by $\mG$ (the reduced Gröbner basis of $I$ for $>_\omega$), respectively.
Let $M = \lcm(M_\alpha,M_\beta)$ be the least common multiple of $M_\alpha$ and $M_\beta$. Then, the $S$-polynomial of $\bfh_\alpha$ and $\bfh_\beta$ is 
\[S_{\alpha,\beta} = S(\bfh_\alpha,\bfh_\beta) = \frac{M}{M_\alpha} \bfh_\alpha - \frac{M}{M_\beta} \bfh_\beta \, .\]
If $S(\bfh_\alpha,\bfh_\beta) = 0$, we are done. Otherwise, note that $\psi_0(S_{\alpha,\beta}) = 0$, so $\bar{\psi}_0 (S_{\alpha,\beta}) \in I$, and hence  $\ini{\bar{\psi}_0 (S_{\alpha,\beta})} \in \ini{I}$. Thus, there exist $c\in \Bbbk$, $w\in \mB_0 \cap J$ and a monomial $M_\mu \in A$ such that $\ini{\bar{\psi}_0 (S_{\alpha,\beta})} = c\cdot w \cdot M_\mu$. Therefore, $M_\mu \in I_w$, and there exists a monomial $M_\gamma \in G(I_w)$ that divides $M_\mu$. Let $\bfh_\gamma \in \mH$ be the element whose initial term is $\ini{\bfh_\gamma} = M_\gamma \bfeps_w$. Consider $S_{\alpha,\beta}' = S_{\alpha,\beta} - c\cdot \frac{M_\mu}{M_\gamma} \bfh_\gamma$. If $S_{\alpha,\beta}' = 0$, we are done. Otherwise, one has that $\bar{\psi}_0(S_{\alpha,\beta}') \in I$ and
$0 \neq \ini{\bar{\psi}_0 (S_{\alpha,\beta}')} <_\omega \ini{\bar{\psi}_0 (S_{\alpha,\beta})}$. We can iterate this process and conclude that $S_{\alpha,\beta}$ reduces to zero modulo $\mH$ by induction because $>_\omega$ is a well order. This shows that $\mH$ is a Gröbner basis of $\ker(\psi_0)$ for $>_{\rm SL}$.

Moreover, since $\bx^\alpha \nmid \bx^\beta$ and $\bx^\beta \nmid \bx^\alpha$ for all $\bx^\alpha \neq \bx^\beta$ in $\mB_1'$, $\mH$ is  minimal. Finally, for each $\bx^\alpha \in \mB_1'$, every monomial appearing in $r_\alpha$ (the remainder of the division of $\bx^\alpha$ by $\mG$), does not belong to $\ini{I}$. Therefore, each monomial that appears in $\sum_{v\in \mB_0} f_{\alpha,v} \cdot \bfeps_v$ does not belong to $\langle \ini{\bfh_\beta} \mid \bfh_\beta \in \mH \rangle = \ini{\ker (\psi_0)}$, and we are done. 
\end{proof}

Since $\mH$ is a Gröbner basis of $\ker(\psi_0)$,
the reductions of the $S$-polynomials $S_{\alpha\beta}$ provide a generating set for the next syzygy module by \cite[Chap.~5, Thm.~3.2]{CLO2}. 
This generating set is indeed a Gröbner basis by Schreyer's Theorem \cite[Thm.~15.10]{Eis95}. The order used here is the Schreyer order induced in
$\oplus_{\bx^\alpha \in \mB_1'} A \left( -\degw(\bx^\alpha) \right)$
by our Schreyer-like order in $\oplus_{v\in \mB_0} A(-\deg_\omega(v))$. 
Applying repeatedly Schreyer's Theorem, we obtain the co-called Schreyer resolution that may not be minimal. Moreover, if we sort at each step the elements of the Gröbner basis as in \cite[Cor.~15.11]{Eis95}, one variable disappears from the initial terms of the elements in the Gröbner basis at each step. Mimicking the proof of Hilbert's Syzygies Theorem that uses iteratedly Schreyer's Theorem, we obtain a $\omega$-graded free resolution of $R/I$ as $A$-module that may not be minimal but has at most $d$ steps,
\begin{equation}\label{eq:NoetherResnomin}
\mathcal{F}': 0 \rightarrow \oplus_{v\in \mB_{p'}'} A(-\deg_\omega(v)) \xrightarrow{\psi_{p'}'} \dots \xrightarrow{\psi_1'} \oplus_{v\in \mB_0} A(-\deg_\omega(v)) \xrightarrow{\psi_0} R/I \rightarrow 0 \, ,
\end{equation}
where $d \geq p'\geq p$ and $\mB_i' \subset R$ is a set of monomials for all $i$.   
Minimalizing this resolution, a short resolution of $R/I$ as in \eqref{eq:NoetherRes} is obtained with $\mB_i \subset \mB_i'$. 

We now illustrate with an example how to build and minimize Schreyer's resolution. We will see later in Section~\ref{sec:pruning} how to explicitly obtain the short resolution \eqref{eq:NoetherRes} from Schereyer's resolution \eqref{eq:NoetherResnomin} when $R/I$ is a simplicial semigroup ring of dimension $3$.

\begin{example} \label{ex:dim2_zerodiv}
Consider the ideal $I\subset R=\Q[x_1,x_2,x_3,x_4,x_5]$ in Example~\ref{ex:dim2_zerodiv_alg1}. 
We have already determined $\mB_0$ and $\mB_1'$, and we now sort the elements in $\mB_1'$ as follows:
\[\begin{split}
\mB_1' = \{ &{} v_{1}=x_{2} x_{3}^{2} x_{4},v_{2}=x_{2} x_{3}^{2} x_{5},v_{3}=x_{1} x_{3}^{2} x_{4},v_{4}=x_{1} x_{3}^{2} x_{5},v_{5}=x_{2} x_{3} x_{4}^{2},v_{6}=x_{2} x_{3} x_{5}^{2},\\
&{} v_{7}=x_{1} x_{3} x_{4}^{2},v_{8}=x_{1} x_{3} x_{4} x_{5}^{2}, v_{9}=x_{1} x_{3} x_{5}^{4},v_{10}=x_{2}^{3} x_{4},v_{11}=x_{2}^{3} x_{5}^{2},v_{12}=x_{1} x_{2}^{2} x_{4}, \\
&{} v_{13}=x_{2}^{2} x_{4}^{2},v_{14}=x_{1} x_{2} x_{4}^{3},v_{15}=x_{1}^{2} x_{4}^{2},v_{16}=x_{1}^{2} x_{5} \} \,.
\end{split}\]
The 16 generators of $\ker(\psi_0)$ given by Proposition~\ref{prop:ker_psi0} are
\begin{center}
\begin{tabular}{lll}
$\bfh_1 = x_4 \bfeps_3 - x_4^3x_5 \bfeps_{16}$,  &  $\bfh_2 = x_5 \bfeps_3 - x_4^2x_5^2 \bfeps_{16}$,  &  $\bfh_3 = x_4 \bfeps_4 - x_5 \bfeps_9 - x_4^2x_5 \bfeps_8 + x_5^3 \bfeps_8$, \\
$\bfh_4 = x_5 \bfeps_4 - x_4 x_5^2 \bfeps_8$,  &  $\bfh_5 = x_4^2 \bfeps_6 - x_4^3 \bfeps_{15}$,  &  $\bfh_6 = x_5^2 \bfeps_6 - x_4 x_5^2 \bfeps_{15}$, \\
\multicolumn{2}{l}{$\bfh_7 = x_4^2 \bfeps_7 - x_4^3x_5 \bfeps_{16} - x_5^2 \bfeps_7 + x_4x_5^3 \bfeps_{16}$} & $\bfh_8 = x_4x_5^2 \bfeps_7 - x_4^2x_5^3 \bfeps_{16}$,\\
$\bfh_9 = x_5^4 \bfeps_7 - x_4x_5^5 \bfeps_{16}$, & $\bfh_{10} = x_4 \bfeps_9 - x_4 x_5^2 \bfeps_8$, &  $\bfh_{11} = x_5^2 \bfeps_{9} - x_5^4 \bfeps_8$,  \\
$\bfh_{12} = x_4 \bfeps_{10} - x_5^2 \bfeps_5$, & $\bfh_{13} = x_4^2 \bfeps_{11} - x_5 \bfeps_2$,  &  $\bfh_{14} = x_4^3 \bfeps_{12} - x_5 \bfeps_1$, \\
$\bfh_{15} = x_4^2 \bfeps_{14} - x_4^2 x_5 \bfeps_{13}$, & $\bfh_{16} = x_5 \bfeps_{14}- x_5^2 \bfeps_{13}$ , & \\
\end{tabular}
\end{center}
where $\{\bfeps_1,\ldots,\bfeps_{16}\}$ denotes the canonical basis of $\oplus_{i=1}^{16} A \left( -\deg(u_i) \right)$.

By Theorem~\ref{them:GBschreyer}, $\mH=\{\bfh_{1},\ldots,\bfh_{16}\}$ is the reduced Gröbner basis of $\ker(\psi_0)$ for our Schreyer-like order $>_{\rm SL}$. Moreover, in the above list, the first term of each element is its initial term by Remark~\ref{rem:initial}.
Note that we have sorted $\bfh_1,\ldots,\bfh_{16}$ 
(and, accordingly, $v_1,\ldots,v_{16}$ in $\mB_1'$)
in such a way that, if for some $i<j$, the initial terms of $\bfh_i$ and $\bfh_j$ are multiples of the same element of the canonical basis, say $\ini{\bfh_i} = M_i \cdot \bfeps_u$ and $\ini{\bfh_j} = M_j \cdot \bfeps_u$ for some $u\in \mB_0$ and two monomials $M_i$ and $M_j$ in $A = \Q[x_4,x_5]$, then $M_i>M_j$ for the lexicographic order $>$ with $x_4>x_5$, i.e., if $M_i=x_4^{a_i}x_5^{b_i}$ and $M_j=x_4^{a_j}x_5^{b_j}$, $a_i > a_j$. This guarantees that $x_4$ will not appear in the leading terms of the generators of the next syzygy module obtained by applying Schreyer's Theorem \cite[Thm.~15.10]{Eis95}, and hence we will be done.

The reductions of the $S$-polynomials $S(\bfh_i,\bfh_j)$ for all $1 \leq i < j \leq 16$, provide a Gröbner basis of the next syzygy module for the induced Schreyer order. Since the only $S$-polynomials that have to be computed and reduced by $\mH$ are the $S(\bfh_i, \bfh_j)$ such that the initial terms of $\bfh_i$ and $\bfh_j$ 
are multiples of the same element in the canonical basis, one just has to focus on 
$S(\bfh_1,\bfh_2)$, $S(\bfh_3,\bfh_4)$, $S(\bfh_5,\bfh_6)$, $S(\bfh_7,\bfh_8)$, $S(\bfh_7,\bfh_9)$, $S(\bfh_8,\bfh_9)$, $S(\bfh_{10},\bfh_{11})$ and $S(\bfh_{15},\bfh_{16})$.
Note that the leading term of the syzygy corresponding to the reduction of $S(\bfh_7,\bfh_9)$ is a multiple of the one coming from $S(\bfh_7,\bfh_8)$, and hence the syzygy coming from $S(\bfh_7,\bfh_9)$ will be discarded when the 
Gröbner basis is minimalized. Thus, we do not need to compute it, and by reducing the other seven $S$-polynomials, we get that the set of monomials $\mB_2'$ is 
\[\begin{split}
\mB_2' = \{ &{}w_{1}=x_{2} x_{3}^{2} x_{4} x_{5},w_{2}=x_{1} x_{3}^{2} x_{4} x_{5},w_{3}=x_{2} x_{3} x_{4}^{2} x_{5}^{2},w_{4}=x_{1} x_{3} x_{4}^{2} x_{5}^{2},w_{5}=x_{1} x_{3} x_{4} x_{5}^{4},\\
&{} w_{6}=x_{2}^{3} x_{4} x_{5}^{2}, w_{7}=x_{1}^{2} x_{4}^{2} x_{5}\} \, .
\end{split}\]
Hence, a graded free resolution of $R/I$ as $A$-module is 
\[
0 \rightarrow  
\oplus_{v\in \mB_2'} A(-\deg(v))
\xrightarrow{\psi_2'}
\oplus_{v\in \mB_1'} A(-\deg(v))
\xrightarrow{\psi_1'}
\oplus_{v\in \mB_0} A(-\deg(v))
\xrightarrow{\psi_0} R/I \rightarrow 0 \, ,
\]
where the matrix of $\psi_0$ is 
$\left(\begin{array}{rrrr}u_1+I & u_2+I &\ldots& u_{16}+I\end{array}\right)$,
the matrix of $\psi_1'$ is the square matrix
$\left(\begin{array}{rrrr}\bfh_1 & \bfh_2 &\ldots& \bfh_{16}\end{array}\right)$,
and the matrix of $\psi_2'$ is given by the reductions of the S-polynomials 
$S(\bfh_1,\bfh_2)$, $S(\bfh_3,\bfh_4)$, $S(\bfh_5,\bfh_6)$, $S(\bfh_7,\bfh_8)$, $S(\bfh_8,\bfh_9)$, $S(\bfh_{10},\bfh_{11})$, and $S(\bfh_{15},\bfh_{16})$.
Since there are nonzero constants in the reduction of 
the second and fourth $S$-polynomials,
\[\begin{split}
S(\bfh_3,\bfh_4) &= x_5\bfh_3-x_4\bfh_4 = -x_5^2 \bfeps_9 + x_5^4 \bfeps_8 = -\bfh_{11} \, ,\\
S(\bfh_7,\bfh_8) &= x_5^2\bfh_7-x_4\bfh_8 = -x_5^4 \bfeps_7 + x_4x_5^5 \bfeps_{16} = -\bfh_9 \, ,
\end{split}\]
the above resolution is not minimal. Making it minimal, we get the short resolution of $R/I$ as in \eqref{eq:NoetherRes} for 
$\mB_1 = \mB_1' \setminus \{v_9,v_{11}\}$ and $\mB_2 = \mB_2' \setminus \{w_2,w_4\}$. 
Reordering, at each step, the generators (and hence the rows and columns of the matrices defining the morphisms),
the short resolution of $R/I$ shows as
{\small\[ 0 \rightarrow  \begin{array}{c} A(-5)^2 \oplus A(-6)^2 \\ \oplus A(-7) \end{array}  
\xrightarrow{\psi_2} \begin{array}{c} A(-3) \oplus A(-4)^{11}  \\ \oplus A(-5)^2  \end{array} 
\xrightarrow{\psi_1} \begin{array}{c} A \oplus A(-1)^3 \oplus A(-2)^6  \\ \oplus A(-3)^5 \oplus A(-4) \end{array}
\xrightarrow{\psi_0} R/I \rightarrow 0 \, ,\]}
and the Betti table is the same as the one given in Example~\ref{ex:dim2_zerodiv_alg1}.
\end{example}

Since the Hilbert series of $R/I$ can be determined using any graded free resolution of $R/I$, we can use the Schreyer resolution \eqref{eq:NoetherResnomin} to compute it, and we get 
\begin{equation}\label{eq:hilbSeries}
\hs{R/I}(t) = \dfrac{\sum_{i=0}^{p'} \sum_{v \in \mB_i'} (-1)^i \, t^{\deg_\omega (v)}}{(1-t)^d} \, .
\end{equation}
As $d = \dim(R/I)$, the numerator does not vanish at $t = 1$ and the expression of the Hilbert series cannot be simplified. As a consequence, one can compute the Hilbert-Samuel multiplicity of $R/I$ (which is the degree of the projective algebraic variety defined by $I$ whenever $I$ is homogeneous) from the size of the sets $\mB_i'$.

\begin{proposition} \label{prop:deg_shortres}
Denote by $e(R/I)$ the {\rm(}Hilbert-Samuel{\rm)} multiplicity of $R/I$. Then,
\begin{enumerate}[(a)]
    \item\label{prop:deg_shortres_a} $e(R/I)=\sum_{i=0}^{p'} (-1)^i |\mB_i'| $.
    \item\label{prop:deg_shortres_b} For all $s\in \N$, \[\sum_{k=0}^d (-1)^k \binom{d}{k} \hf{R/I}(s-k) = \sum_{i=0}^{p'} (-1)^i |(\mB_i')_s| \, , \] where $(\mB_i')_s := \{v \in \mB_i' \mid \deg_\omega(v) = s\}$.
\end{enumerate}
\end{proposition}

\begin{proof}
Evaluating the numerator of the Hilbert series \eqref{eq:hilbSeries} in $t=1$, we obtain~\ref{prop:deg_shortres_a}. For~\ref{prop:deg_shortres_b}, note that  
\[(1-t)^d \hs{R/I}(t) = \sum_{i=0}^{p'} \sum_{v \in \mB_i'} (-1)^i \, t^{\deg_\omega (v)} \] 
and compare the coefficient of $t^s$ in both sides of the equality.
\end{proof}

Although the Schreyer resolution \eqref{eq:NoetherResnomin} is not minimal in general, there are cases in which it is known to be: when $R/I$ is Cohen-Macaulay (\cite[Prop.~1]{BGGM2017}), when $\dim(R/I) =1$ (\cite[Prop.~3]{BGGM2017}), or when $\dim(R/I) = 2$ and $x_n$ is not a zero divisor of $R/I$ (\cite[Prop.~4]{BGGM2017}). The following straightforward result provides another case in which it is minimal.

\begin{proposition}
If, for all $u \in \mB_0 \cap J$,
the monomial ideal $I_u = (\ini{I} : u ) \cap A$ is principal, then
\[0 \rightarrow \oplus_{v\in \mB_1'} A(-\deg_\omega(v)) \xrightarrow{\psi_1'} \oplus_{v\in \mB_0} A(-\deg_\omega(v)) \xrightarrow{\psi_0} R/I \rightarrow 0 \]
is the short resolution of $R/I$, i.e., it is the minimal graded free resolution of $R/I$ as $A$-module. In particular, ${\rm depth}(R/I) = d-1$.
\end{proposition}

The condition in the previous result is not necessary and one can have $\depth{R/I} = d-1$ when $I_u$ is not principal for some $u \in \mB_0 \cap J$ as the following example shows.

\begin{example}\label{ex:noprincipalnominimal}
Consider $R = \Q[x_1,\ldots,x_7]$, $A = \Q[x_5,x_6,x_7]$, and $R/I$, the $3$-dimensional simplicial toric ring defined by 
$\mA = \{(1,3,5), (5,1,5), (3,5,3), (5,5,1)$, $(2,0,0), (0,2,0), (0,0,2)\}$,
i.e., 
\[I = \id{x_1-t_1 t_2^3 t_3^5, x_2-t_1^5 t_2 t_3^5, x_3-t_1^3t_2^5t_3^3,x_4-t_1^5t_2^5t_3,x_5-t_1^2,x_6-t_2^2,x_7-t_3^2} \cap \Q[x_1,\ldots,x_7] \, .\]
Variables are in Noether position, and $I$ is $\omega$-homogeneous for $\omega = (9,11,11,11,2,2,2)$.
One can check, using for example \cite{Sage}, that $\mB_0 = \{x_4,x_3,x_2,x_1,1\}$, $\mB_0 \cap J = \{x_3,x_2,x_1\}$, $I_{x_3} = \id{x_5}$, $I_{x_2} = \id{x_6^2}$ and $I_{x_1} = \id{x_5^2,x_5x_6}$. Hence, $\mB_1' = \{x_3x_5,x_2x_6^2,x_1x_5^2,x_1x_5x_6\}$ and $\mB_2' = \{ x_1x_5^2x_6 \}$. However, the $\omega$-graded short resolution of $R/I$, which can be computed using the function {\tt shortRes} of \cite{github_shortres}, is \[0 \rightarrow A(-13)^3 \rightarrow A \oplus A(-9) \oplus A(-11)^3 \rightarrow R/I \rightarrow 0 \, ,\]
so $|\mB_1| = 3$ and $|\mB_2|=0$.
Therefore, ${\rm pd}_{A}(R/I) = 1$ and $\depth{R/I} = d-1$, although $I_{x_1}$ is not principal.
\end{example}

\section{Simplicial toric rings of dimension 3} \label{sec:homology}

Let $\mA = \{\bfa_1,\ldots,\bfa_n\} \subset \N^d$ be a finite set of nonzero vectors and $\mS   \subset \N^d$ the affine monoid spanned by $\mA$. We suppose that the semigroup $\mS$ is simplicial, and hence we can assume without loss of generality that $\bfa_{n-d+i} = \omega_{n-d+i} \bfeps_i$ for all $i\in \{1,\ldots,d\}$, where $\{\bfeps_1,\ldots,\bfeps_d\}$ denotes the canonical basis of $\N^d$, and $\omega_{n-d+1},\ldots,\omega_n \in \mathbb Z^+$.
Set $\bfe_i = \bfa_{n-d+i}$ for all $i\in \{1,\ldots,d\}$ and denote by $\mE = \{\bfe_1,\ldots,\bfe_d\}$ the extremal rays of the rational cone spanned by $\mA$.

\subsection{A combinatorial description of the short resolution} 

\ \vspace{1mm} 

The simplicial semigroup ring $\ks$ is an $\mS$-graded $\Bbbk$-algebra isomorphic to $R/I_\mA$, where $I_\mA$ is the toric ideal defined by $\mA$. 
We study here the multigraded short resolution of $\ks$ with respect to the multigrading $\degs{x_i} = \deg_\mS(x_i) = \bfa_i \in \mS$; namely, 
\begin{equation}
\mathcal{F}: 0 \rightarrow \oplus_{\bfs\in \mS_p} A (-\bfs) \xrightarrow{\psi_p} \dots \xrightarrow{\psi_1} \oplus_{\bfs\in \mS_0} A (-\bfs) \xrightarrow{\psi_0} \ks \rightarrow 0\, ,
\label{eq:NoetherRes_semigroup}
\end{equation}
where $\mS_i \subset \mS$ is a multiset for all $i\in \{0,\ldots,p\}$. Note that this multigrading is a refinement of the grading given by the weight vector $\omega = (\omega_1,\ldots,\omega_n)$, where $\omega_i = |\bfa_i| = \sum_{j=1}^d a_{ij} \in \Z^+$ for all $i \in \{1,\ldots,n\}$. Hence, $I_\mA$ is \whom{} and the results of Section~\ref{sec:groebner} apply here. As in that section, we fix the $\omega$-graded reverse lexicographic order $>_\omega$ in $R$.
\newline

Our goal here is to describe the resolution \eqref{eq:NoetherRes_semigroup} in terms of the combinatorics of the semigroup $\mS$ when $d=3$, i.e., when the Krull dimension of $\ks$ is $3$. 
We start by recalling from \cite{BGGM2017} the first step of the resolution for any $d\geq 1$.
We will later describe the multisets $\mS_1,\mS_2 \subset \mS$ that appear in the short resolution \eqref{eq:NoetherRes_semigroup} of $\ks$. This is a combinatorial transcription of results in \cite{Pison2003} and \cite{Ojeda2017} that will be useful in Section~\ref{sec:pruning}.

\begin{definition}
Let $\mS$ be a simplicial semigroup and denote by $\mE = \{\bfe_1,\ldots,\bfe_d\}$ the set of extremal rays of the rational cone spanned by $\mA$.
The Apery set of $\mS$ is \[\aps := \{\bfs \in \mS: \bfs - \bfe_i \notin \mS \text{ for all } i = 1,\ldots,d \}\, .\]
\end{definition}

\begin{proposition}[{\cite[Prop.~5]{BGGM2017}}]
The set $\mS_0$ in the short resolution \eqref{eq:NoetherRes_semigroup} is $\aps$, the Apery set of $\mS$. 
The $\mS$-graded $A$-module homomorphism $\psi_0: \oplus_{\bfs\in \mS_0} A(-\bfs) \rightarrow \ks$ is defined by $\psi_0(\bfeps_\bfs) = t^\bfs$, where $\{\bfeps_\bfs \mid \bfs \in \mS_0\}$ is the canonical basis of $\oplus_{\bfs \in \mS_0} A(-\bfs)$. 
\end{proposition}

To compute the multidegrees in the next steps of the resolution, we consider, for every $\bfs \in \mS$, the abstract simplicial complex $T_\bfs$ defined by 
\[T_\bfs := \left\{ \mF \subset \mE : \bfs - \sum_{\bfe \in \mF} \bfe \in \mS \right\} \,.\]
In \cite[Prop.~2.1]{Pison2003} and \cite[Prop.~5.1]{Ojeda2017}, the authors prove that the number of syzygies of multidegree $\bfs$ at the $(i+1)$-th step of the minimal $\mS$-graded resolution \eqref{eq:NoetherRes_semigroup}
is $\dim_\k{\tilde{H}_i(T_\bfs)}$, where $\tilde{H}_i(-)$ denotes the $i$-th reduced homology $\Bbbk$-vector space of $T_\bfs$. 
\newline

If $\bfs \in \mS$ is such that $\bfs- \sum_{\bfe \in \mE} \bfe \in \mS$, then $T_\bfs$ is a simplex and $\dim_\k{\tilde{H}_i(T_\bfs)}=0$ for all $i\in\Z$. Hence, such an element $\bfs \in \mS$ does not belong to any of the multisets $\mS_i$ in \eqref{eq:NoetherRes_semigroup}. We are thus interested in the elements $\bfs \in \mS$ such that $\bfs- \sum_{\bfe \in \mE} \bfe \notin \mS$ which we will classify. When $d=3$, we now define four exceptional sets that will play a role in our combinatorial description. The rest of the elements in $\mS$, which will not contribute in any multiset $\mS_i$ in \eqref{eq:NoetherRes_semigroup}, appear in Table~\ref{tab:simp_complexes} as {\it Other configurations}.
Recall that, when $d=3$, $\mE = \{\bfe_1,\bfe_2,\bfe_3\}$. In the notation $E_\mS^{a,b}$ below, $a$ is the number of indices $i$, $1\leq i\leq 3$, such that $\bfs-\bfe_i \in \mS$ and $b$ is the number of pairs $(i,j)$ of indices,  $1\leq i < j\leq 3$, such that $\bfs-\bfe_i-\bfe_j \in \mS$ (according to this notation, the Apery set defined before would be $E_\mS^{0,0}$). Note that $0\leq b \leq a \leq 3$.

\begin{definition}
Let $\mS \subset \N^3$ be a simplicial semigroup.
We define the following subsets of $\mS$, which we call the \textit{exceptional sets of $\mS$}:
\begin{itemize}
    \item 
    $E_\mS^{3,1} = \{\bfs \in \mS \mid \bfs-\bfe_j \in \mS, \forall j; \, \bfs-(\bfe_{i_1}+\bfe_{i_2}) \in \mS, \bfs-(\bfe_{i_1}+\bfe_{i_3}) \notin \mS, \bfs-(\bfe_{i_2}+\bfe_{i_3}) \notin \mS, \text{ for a permutation } (i_1,i_2,i_3) \text{ of } (1,2,3)\}$;
    
    \item 
    $E_\mS^{2,0} = \{\bfs \in \mS \mid \bfs-\bfe_{i_1} \in \mS, \bfs-\bfe_{i_2} \in \mS, \bfs - \bfe_{i_3} \notin \mS ; \, \bfs-(\bfe_{i_1}+\bfe_{i_2}) \notin \mS, \text{ for a permutation } (i_1,i_2,i_3) \text{ of } (1,2,3)\}$;
    
    \item $E_\mS^{3,0} = \{\bfs \in \mS \mid \bfs - \bfe_i \in \mS, \forall i ; \, \bfs-(\bfe_i+\bfe_j) \notin \mS, \forall i\neq j\}$;
    
    \item $E_\mS^{3,3} = \{\bfs \in \mS \mid \bfs - \bfe_i \in \mS, \forall i ; \, \bfs-(\bfe_i+\bfe_j) \in \mS, \forall i\neq j; \, \text{and } \bfs-(\bfe_1+\bfe_2+\bfe_3) \notin \mS\}$.
\end{itemize}
\end{definition}

Figure~\ref{fig:aps_except} shows how elements in the Apery and the exceptional sets of $\mS$ look like. In those figures, filled circles represent elements in $\mS$, while empty squares represent elements outside $\mS$.

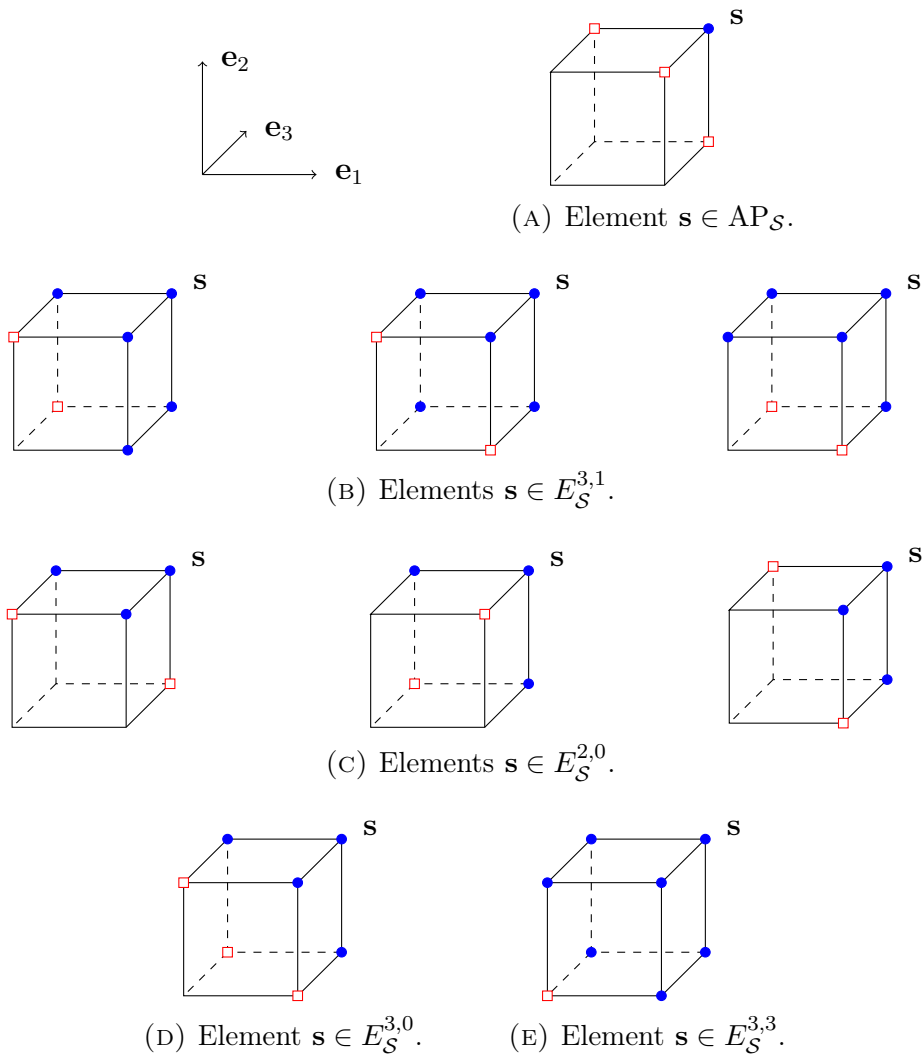
\begin{figure}
\centering
\begin{subfigure}[c]{0.3\linewidth}
\centering
\begin{tikzpicture}[scale=1.5]
  \draw[thin,->](0,0,0)-- (1,0,0) node[right=1mm]{$\mathbf{e}_1$};
  \draw[thin,->](0,0,0)--(0,1,0) node[right=1mm]{$\mathbf{e}_2$};
  \draw[thin,->](0,0,0)--(0,0,-1) node[right=1mm]{$\mathbf{e}_3$};
\end{tikzpicture}    
\end{subfigure}
\begin{subfigure}[c]{0.3\linewidth}
\centering
\begin{tikzpicture}[scale=1.5]
  \node at (1.25,1.1,0) {$\bfs$};
  \draw[thin](1,1,0)--(0,1,0)--(0,1,1)--(1,1,1)--(1,1,0)--(1,0,0)--(1,0,1)--(0,0,1)--(0,1,1);
  \draw[thin](1,1,1)--(1,0,1);
  \draw[thin,dashed](1,0,0)--(0,0,0)--(0,1,0);
  \draw[thin,dashed](0,0,0)--(0,0,1);
  \draw [red] plot [only marks, mark size=1.2, mark=square*, mark options = {fill=white}] coordinates {(1,1,1) (1,0,0) (0,1,0)};
  \draw [blue] plot [only marks, mark size=1.2, mark=*] coordinates {(1,1,0)};
\end{tikzpicture}
\caption{Element $\bfs \in \aps$.}
\end{subfigure}
\\[\bigskipamount]
\begin{subfigure}[b]{\linewidth}
\centering
\begin{tikzpicture}[scale=1.5]
\node at (1.25,1.1,0) {$\bfs$};
  \draw[thin](1,1,0)--(0,1,0)--(0,1,1)--(1,1,1)--(1,1,0)--(1,0,0)--(1,0,1)--(0,0,1)--(0,1,1);
  \draw[thin](1,1,1)--(1,0,1);
  \draw[thin,dashed](1,0,0)--(0,0,0)--(0,1,0);
  \draw[thin,dashed](0,0,0)--(0,0,1);
  \draw [red] plot [only marks, mark size=1.2, mark=square*, mark options = {fill=white}] coordinates {(0,1,1) (0,0,0)};
  \draw [blue] plot [only marks, mark size=1.2, mark=*] coordinates {(1,1,0) (0,1,0) (1,0,0) (1,1,1) (1,0,1)};
\end{tikzpicture}
\hspace{1.75cm}
\begin{tikzpicture}[scale=1.5]
\node at (1.25,1.1,0) {$\bfs$};
  \draw[thin](1,1,0)--(0,1,0)--(0,1,1)--(1,1,1)--(1,1,0)--(1,0,0)--(1,0,1)--(0,0,1)--(0,1,1);
  \draw[thin](1,1,1)--(1,0,1);
  \draw[thin,dashed](1,0,0)--(0,0,0)--(0,1,0);
  \draw[thin,dashed](0,0,0)--(0,0,1);
  \draw [red] plot [only marks, mark size=1.2, mark=square*, mark options = {fill=white}] coordinates {(0,1,1) (1,0,1)};
  \draw [blue] plot [only marks, mark size=1.2, mark=*] coordinates {(1,1,0) (0,1,0) (1,0,0) (1,1,1) (0,0,0)};
\end{tikzpicture}
\hspace{1.6cm}
\begin{tikzpicture}[scale=1.5]
\node at (1.25,1.1,0) {$\bfs$};
  \draw[thin](1,1,0)--(0,1,0)--(0,1,1)--(1,1,1)--(1,1,0)--(1,0,0)--(1,0,1)--(0,0,1)--(0,1,1);
  \draw[thin](1,1,1)--(1,0,1);
  \draw[thin,dashed](1,0,0)--(0,0,0)--(0,1,0);
  \draw[thin,dashed](0,0,0)--(0,0,1);
  \draw [red] plot [only marks, mark size=1.2, mark=square*, mark options = {fill=white}] coordinates {(0,0,0) (1,0,1)};
  \draw [blue] plot [only marks, mark size=1.2, mark=*] coordinates {(1,1,0) (0,1,0) (1,0,0) (1,1,1) (0,1,1)};
\end{tikzpicture}
\caption{Elements $\bfs \in E_\mS^{3,1}$.}
\end{subfigure}
\\[\bigskipamount]
\begin{subfigure}[b]{\linewidth}
\centering
\begin{tikzpicture}[scale=1.5]
\node at (1.25,1.1,0) {$\bfs$};
  \draw[thin](1,1,0)--(0,1,0)--(0,1,1)--(1,1,1)--(1,1,0)--(1,0,0)--(1,0,1)--(0,0,1)--(0,1,1);
  \draw[thin](1,1,1)--(1,0,1);
  \draw[thin,dashed](1,0,0)--(0,0,0)--(0,1,0);
  \draw[thin,dashed](0,0,0)--(0,0,1);
  \draw [red] plot [only marks, mark size=1.2, mark=square*, mark options = {fill=white}] coordinates {(0,1,1) (1,0,0)};
  \draw [blue] plot [only marks, mark size=1.2, mark=*] coordinates {(1,1,0) (0,1,0) (1,1,1)};
\end{tikzpicture}
\hspace{1.75cm}
\begin{tikzpicture}[scale=1.5]
\node at (1.25,1.1,0) {$\bfs$};
  \draw[thin](1,1,0)--(0,1,0)--(0,1,1)--(1,1,1)--(1,1,0)--(1,0,0)--(1,0,1)--(0,0,1)--(0,1,1);
  \draw[thin](1,1,1)--(1,0,1);
  \draw[thin,dashed](1,0,0)--(0,0,0)--(0,1,0);
  \draw[thin,dashed](0,0,0)--(0,0,1);
  \draw [red] plot [only marks, mark size=1.2, mark=square*, mark options = {fill=white}] coordinates {(0,0,0) (1,1,1)};
  \draw [blue] plot [only marks, mark size=1.2, mark=*] coordinates {(1,1,0) (0,1,0) (1,0,0)};
\end{tikzpicture}
\hspace{1.75cm}
\begin{tikzpicture}[scale=1.5]
\node at (1.25,1.1,0) {$\bfs$};
  \draw[thin](1,1,0)--(0,1,0)--(0,1,1)--(1,1,1)--(1,1,0)--(1,0,0)--(1,0,1)--(0,0,1)--(0,1,1);
  \draw[thin](1,1,1)--(1,0,1);
  \draw[thin,dashed](1,0,0)--(0,0,0)--(0,1,0);
  \draw[thin,dashed](0,0,0)--(0,0,1);
  \draw [red] plot [only marks, mark size=1.2, mark=square*, mark options = {fill=white}] coordinates {(0,1,0) (1,0,1)};
  \draw [blue] plot [only marks, mark size=1.2, mark=*] coordinates {(1,1,0) (1,1,1) (1,0,0)};
\end{tikzpicture}
\caption{Elements $\bfs \in E_\mS^{2,0}$.}
\end{subfigure}
\\[\bigskipamount]
\centering
\begin{subfigure}[b]{0.3\linewidth}
\centering
\begin{tikzpicture}[scale=1.5]
\node at (1.25,1.1,0) {$\bfs$};
  \draw[thin](1,1,0)--(0,1,0)--(0,1,1)--(1,1,1)--(1,1,0)--(1,0,0)--(1,0,1)--(0,0,1)--(0,1,1);
  \draw[thin](1,1,1)--(1,0,1);
  \draw[thin,dashed](1,0,0)--(0,0,0)--(0,1,0);
  \draw[thin,dashed](0,0,0)--(0,0,1);
  \draw [red] plot [only marks, mark size=1.2, mark=square*, mark options = {fill=white}] coordinates {(0,0,0) (0,1,1) (1,0,1)};
  \draw [blue] plot [only marks, mark size=1.2, mark=*] coordinates {(1,0,0) (0,1,0) (1,1,0) (1,1,1)};
\end{tikzpicture}
\caption{Element $\bfs \in E_\mS^{3,0}$.}
\end{subfigure}
\begin{subfigure}[b]{0.3\linewidth}
\centering
\begin{tikzpicture}[scale=1.5]
\node at (1.25,1.1,0) {$\bfs$};
  \draw[thin](1,1,0)--(0,1,0)--(0,1,1)--(1,1,1)--(1,1,0)--(1,0,0)--(1,0,1)--(0,0,1)--(0,1,1);
  \draw[thin](1,1,1)--(1,0,1);
  \draw[thin,dashed](1,0,0)--(0,0,0)--(0,1,0);
  \draw[thin,dashed](0,0,0)--(0,0,1);
  \draw [red] plot [only marks, mark size=1.2, mark=square*, mark options = {fill=white}] coordinates {(0,0,1)};
  \draw [blue] plot [only marks, mark size=1.2, mark=*] coordinates {(0,0,0) (1,0,0) (0,1,0) (1,1,0) (1,0,1) (0,1,1) (1,1,1)};
\end{tikzpicture}
\caption{Element $\bfs \in E_\mS^{3,3}$.}
\end{subfigure}
\hfill
\caption{Points in $\aps$ and the exceptional sets $E_\mS^{3,1}$, $E_\mS^{2,0}$, $E_\mS^{3,0}$, and $E_\mS^{3,3}$.}
\label{fig:aps_except}
\end{figure}

\begin{theorem} \label{thm:dim3_Si}
If $\ks$ is a simplicial semigroup ring of Krull dimension $d=3$, the multisets $\mS_0,\mS_1,\mS_2 \subset \mS$ that appear in the short resolution \eqref{eq:NoetherRes_semigroup} are 
\[\mS_0 = \aps, \qquad \mS_1 = E_\mS^{3,1} \cup E_\mS^{2,0} \cup E_\mS^{3,0} \cup E_\mS^{3,0}, \qquad  \mS_2 = E_\mS^{3,3} \, .\]
\end{theorem}

\begin{proof}
We already know that $\mS_0 = \aps$. For any other $\bfs \in \mS$, the simplicial complex $T_\bfs$ is one of those in Table~\ref{tab:simp_complexes}, whose homologies are straightforward to compute. Then, the result follows from \cite[Prop.~2.1]{Pison2003} and \cite[Prop.~5.1]{Ojeda2017}
\end{proof}

\begin{table}
\begin{tabular}{cccc|cccc}
\toprule
$E_\mS^{3,1}$ & $E_\mS^{2,0}$ & $E_\mS^{3,0}$ & $E_\mS^{3,3}$ & \multicolumn{4}{c}{Other configurations}\\
\midrule
\begin{tikzpicture}[scale=1]
  \draw[thin](1,1,0)--(0,1,0)--(0,1,1)--(1,1,1)--(1,1,0)--(1,0,0)--(1,0,1)--(0,0,1)--(0,1,1);
  \draw[thin](1,1,1)--(1,0,1);
  \draw[thin,dashed](1,0,0)--(0,0,0)--(0,1,0);
  \draw[thin,dashed](0,0,0)--(0,0,1);
  \draw [red] plot [only marks, mark size=1.2, mark=square*, mark options = {fill=white}] coordinates {(0,1,1) (0,0,0)};
  \draw [blue] plot [only marks, mark size=1.2, mark=*] coordinates {(1,1,0) (0,1,0) (1,0,0) (1,1,1) (1,0,1)};
\end{tikzpicture}
& 
\begin{tikzpicture}[scale=1]
\draw[thin](1,1,0)--(0,1,0)--(0,1,1)--(1,1,1)--(1,1,0)--(1,0,0)--(1,0,1)--(0,0,1)--(0,1,1);
\draw[thin](1,1,1)--(1,0,1);
\draw[thin,dashed](1,0,0)--(0,0,0)--(0,1,0);
\draw[thin,dashed](0,0,0)--(0,0,1);
\draw [red] plot [only marks, mark size=1.2, mark=square*, mark options = {fill=white}] coordinates {(0,1,1) (1,0,0)};
\draw [blue] plot [only marks, mark size=1.2, mark=*] coordinates {(1,1,0) (0,1,0) (1,1,1)};
\end{tikzpicture}
& 
\begin{tikzpicture}[scale=1]
  \draw[thin](1,1,0)--(0,1,0)--(0,1,1)--(1,1,1)--(1,1,0)--(1,0,0)--(1,0,1)--(0,0,1)--(0,1,1);
  \draw[thin](1,1,1)--(1,0,1);
  \draw[thin,dashed](1,0,0)--(0,0,0)--(0,1,0);
  \draw[thin,dashed](0,0,0)--(0,0,1);
  \draw [red] plot [only marks, mark size=1.2, mark=square*, mark options = {fill=white}] coordinates {(0,0,0) (0,1,1) (1,0,1)};
  \draw [blue] plot [only marks, mark size=1.2, mark=*] coordinates {(1,0,0) (0,1,0) (1,1,0) (1,1,1)};
\end{tikzpicture}
&
\begin{tikzpicture}[scale=1]
  \draw[thin](1,1,0)--(0,1,0)--(0,1,1)--(1,1,1)--(1,1,0)--(1,0,0)--(1,0,1)--(0,0,1)--(0,1,1);
  \draw[thin](1,1,1)--(1,0,1);
  \draw[thin,dashed](1,0,0)--(0,0,0)--(0,1,0);
  \draw[thin,dashed](0,0,0)--(0,0,1);
  \draw [red] plot [only marks, mark size=1.2, mark=square*, mark options = {fill=white}] coordinates {(0,0,1)};
  \draw [blue] plot [only marks, mark size=1.2, mark=*] coordinates {(0,0,0) (1,0,0) (0,1,0) (1,1,0) (1,0,1) (0,1,1) (1,1,1)};
\end{tikzpicture}
&
\begin{tikzpicture}[scale=1]
  \draw[thin](1,1,0)--(0,1,0)--(0,1,1)--(1,1,1)--(1,1,0)--(1,0,0)--(1,0,1)--(0,0,1)--(0,1,1);
  \draw[thin](1,1,1)--(1,0,1);
  \draw[thin,dashed](1,0,0)--(0,0,0)--(0,1,0);
  \draw[thin,dashed](0,0,0)--(0,0,1);
  \draw [red] plot [only marks, mark size=1.2, mark=square*, mark options = {fill=white}] coordinates {(0,1,1) (1,1,1) (0,0,0) (1,0,0)};
  \draw [blue] plot [only marks, mark size=1.2, mark=*] coordinates {(0,1,0) (1,1,0)};
\end{tikzpicture}
&
\begin{tikzpicture}[scale=1]
  \draw[thin](1,1,0)--(0,1,0)--(0,1,1)--(1,1,1)--(1,1,0)--(1,0,0)--(1,0,1)--(0,0,1)--(0,1,1);
  \draw[thin](1,1,1)--(1,0,1);
  \draw[thin,dashed](1,0,0)--(0,0,0)--(0,1,0);
  \draw[thin,dashed](0,0,0)--(0,0,1);
  \draw [red] plot [only marks, mark size=1.2, mark=square*, mark options = {fill=white}] coordinates {(0,0,0) (1,0,0) (0,0,1) (1,0,1)};
  \draw [blue] plot [only marks, mark size=1.2, mark=*] coordinates {(0,1,0) (1,1,0) (0,1,1) (1,1,1) };
\end{tikzpicture}
&
\begin{tikzpicture}[scale=1]
  \draw[thin](1,1,0)--(0,1,0)--(0,1,1)--(1,1,1)--(1,1,0)--(1,0,0)--(1,0,1)--(0,0,1)--(0,1,1);
  \draw[thin](1,1,1)--(1,0,1);
  \draw[thin,dashed](1,0,0)--(0,0,0)--(0,1,0);
  \draw[thin,dashed](0,0,0)--(0,0,1);
  \draw [red] plot [only marks, mark size=1.2, mark=square*, mark options = {fill=white}] coordinates {(0,0,1) (1,0,1)};
  \draw [blue] plot [only marks, mark size=1.2, mark=*] coordinates {(0,1,0) (1,1,0) (0,1,1) (1,1,1) (0,0,0) (1,0,0)};
\end{tikzpicture}
&
\begin{tikzpicture}[scale=1]
  \draw[thin](1,1,0)--(0,1,0)--(0,1,1)--(1,1,1)--(1,1,0)--(1,0,0)--(1,0,1)--(0,0,1)--(0,1,1);
  \draw[thin](1,1,1)--(1,0,1);
  \draw[thin,dashed](1,0,0)--(0,0,0)--(0,1,0);
  \draw[thin,dashed](0,0,0)--(0,0,1);
  \draw [blue] plot [only marks, mark size=1.2, mark=*] coordinates {(0,1,0) (1,1,0) (0,1,1) (1,1,1) (0,0,0) (1,0,0) (0,0,1) (1,0,1)};
\end{tikzpicture}
\\
\midrule
\begin{tikzpicture}[scale=0.9]
\coordinate [label={[xshift=-2mm, yshift=-5mm]:$i_1$}] (i1) at (0,0);
\coordinate [label={[xshift=-2mm, yshift=-5mm]:$i_2$}] (i2) at (1,0);
\coordinate [label={[xshift=-2mm, yshift=0mm]:$i_3$}] (i3) at (0.5,0.866);
\draw [black] plot [only marks, mark size=1.2, mark=*] coordinates {(i1) (i2) (i3)};
\draw[thin] (i1) -- (i2);
\end{tikzpicture}
&
\begin{tikzpicture}[scale=0.9]
\coordinate [label={[xshift=-2mm, yshift=-5mm]:$i_1$}] (i1) at (0,0);
\coordinate [label={[xshift=-2mm, yshift=-5mm]:$i_2$}] (i2) at (1,0);
\draw [black] plot [only marks, mark size=1.2, mark=*] coordinates {(i1) (i2)};
\end{tikzpicture}
&
\begin{tikzpicture}[scale=0.9]
\coordinate [label={[xshift=-2mm, yshift=-5mm]:$i_1$}] (i1) at (0,0);
\coordinate [label={[xshift=-2mm, yshift=-5mm]:$i_2$}] (i2) at (1,0);
\coordinate [label={[xshift=-2mm, yshift=0mm]:$i_3$}] (i3) at (0.5,0.866);
\draw [black] plot [only marks, mark size=1.2, mark=*] coordinates {(i1) (i2) (i3)};
\end{tikzpicture}
&
\begin{tikzpicture}[scale=0.9]
\coordinate [label={[xshift=-2mm, yshift=-5mm]:$i_1$}] (i1) at (0,0);
\coordinate [label={[xshift=-2mm, yshift=-5mm]:$i_2$}] (i2) at (1,0);
\coordinate [label={[xshift=-2mm, yshift=0mm]:$i_3$}] (i3) at (0.5,0.866);
\draw [black] plot [only marks, mark size=1.2, mark=*] coordinates {(i1) (i2) (i3)};
\draw[thin] (i1) -- (i2) -- (i3) -- (i1);
\end{tikzpicture}
&
\begin{tikzpicture}[scale=0.9]
\coordinate [label={[xshift=-2mm, yshift=-5mm]:$i_1$}] (i1) at (0,0);
\draw [black] plot [only marks, mark size=1.2, mark=*] coordinates {(i1)};
\end{tikzpicture}
&
\begin{tikzpicture}[scale=0.9]
\coordinate [label={[xshift=-2mm, yshift=-5mm]:$i_1$}] (i1) at (0,0);
\coordinate [label={[xshift=-2mm, yshift=-5mm]:$i_2$}] (i2) at (1,0);
\draw [black] plot [only marks, mark size=1.2, mark=*] coordinates {(i1) (i2)};
\draw[thin] (i1) -- (i2);
\end{tikzpicture}
&
\begin{tikzpicture}[scale=0.9]
\coordinate [label={[xshift=-2mm, yshift=-5mm]:$i_1$}] (i1) at (0,0);
\coordinate [label={[xshift=-2mm, yshift=-5mm]:$i_2$}] (i2) at (1,0);
\coordinate [label={[xshift=-2mm, yshift=0mm]:$i_3$}] (i3) at (0.5,0.866);
\draw [black] plot [only marks, mark size=1.2, mark=*] coordinates {(i1) (i2) (i3)};
\draw[thin] (i1) -- (i2) -- (i1) -- (i3);
\end{tikzpicture}
&
\begin{tikzpicture}[scale=0.9]
\coordinate [label={[xshift=-2mm, yshift=-5mm]:$i_1$}] (i1) at (0,0);
\coordinate [label={[xshift=-2mm, yshift=-5mm]:$i_2$}] (i2) at (1,0);
\coordinate [label={[xshift=-2mm, yshift=0mm]:$i_3$}] (i3) at (0.5,0.866);
\fill[fill=gray!20] (i1) -- (i2) -- (i3);
\draw[thin] (i1) -- (i2) -- (i3) -- (i1);
\draw [black] plot [only marks, mark size=1.2, mark=*] coordinates {(i1) (i2) (i3)};
\end{tikzpicture}
\\
\bottomrule
\end{tabular}
\caption{Possible configurations of elements $\bfs \in \mS$ and the associated simplicial complexes $T_\bfs$.}
\label{tab:simp_complexes}
\end{table}

\begin{remark}
As a consequence of Theorem~\ref{thm:dim3_Si}, the sets $\aps$, $E_\mS^{3,1}$, $E_\mS^{2,0}$, $E_\mS^{3,0}$, and $E_\mS^{3,3}$ are finite subsets of $\mS$.
\end{remark}

The Apery and exceptional sets of $\mS$ determine the multigraded Hilbert series of $\ks$, as the following result shows.

\begin{corollary}
Let $\ks$ be a simplicial semigroup ring of Krull dimension $3$. The multigraded Hilbert series of $\ks$ is:
\[\hs{\ks}(\bft) = \dfrac{\sum_{\bfs \in \aps} \bft^\bfs - \sum_{\bfs \in E_\mS^{3,1}} \bft^\bfs - \sum_{\bfs \in E_\mS^{2,0}} \bft^\bfs - 2 \sum_{\bfs \in E_\mS^{3,0}} \bft^\bfs +  \sum_{\bfs \in E_\mS^{3,3}} \bft^\bfs}{(1-t_1^{\,\omega_{n-2}}) (1-t_2^{\,\omega_{n-1}}) (1-t_3^{\,\omega_n})} \, ,\]
where $\omega_{n-2} = |\bfe_1|$, $\omega_{n-1} = |\bfe_2|$ and $\omega_n = |\bfe_3|$.
\end{corollary}

\begin{proof}
The multigraded Hilbert series of $\ks$ is given by 
\[\hs{\ks}(\bft) = \sum_{\bfs = (s_1,s_2,s_3) \in \mS} t_1^{s_1}t_2^{s_2}t_3^{s_3} = \dfrac{\sum_{\bfs \in \mS_0} \bft^\bfs - \sum_{\bfs \in \mS_1} \bft^\bfs + \sum_{\bfs \in \mS_2} \bft^\bfs}{(1-t_1^{\,\omega_{n-2}}) (1-t_2^{\,\omega_{n-1}}) (1-t_3^{\,\omega_n})} \, ,\]
and the result follows from Theorem~\ref{thm:dim3_Si}.
\end{proof}

As already observed at the beginning of this section, the ideal $I_\mA$ is \whom{} for the weight vector $\omega = (\omega_1,\ldots,\omega_n)$, where $\omega_i = |\bfa_i|$ for $i\in \{1,\ldots,n\}$. Therefore, the short resolution of $\ks$ with respect to this grading can be obtained from the multigraded one in a simple way as follows:
\[\mathcal{F}: 0 \rightarrow \oplus_{\bfs\in \mS_2} A (-|\bfs|) \xrightarrow{\psi_2} \oplus_{\bfs\in \mS_1} A(-|\bfs|) \xrightarrow{\psi_1} \oplus_{\bfs\in \mS_0} A (-|\bfs|) \xrightarrow{\psi_0} \ks \rightarrow 0\, .\]
Moreover, the weighted Hilbert series of $\ks$ is obtained from the multigraded one by the transformation $t_1^{a_1}t_2^{a_2}t_3^{a_3} \mapsto t^{a_1+a_2+a_3}$. \newline

\subsection{Projective monomial surfaces}

\ \vspace{1mm} 

Assume here that $d=3$ and that there exists $D\in \Z^+$, such that $|\bfa_i| = D$ for all $i=1,\ldots,n$. This implies that the toric ideal determined by $\mA$, $I_\mA$, is homogeneous for the standard grading (see, e.g., \cite[Lem.~4.14]{Sturm}). In particular, $\bfa_{n-2} = D\bfeps_1$, $\bfa_{n-1} = D\bfeps_2$, and $\bfa_{n} = D\bfeps_3$, where $\{\bfeps_1,\bfeps_2,\bfeps_3\}$ denotes the canonical basis of $\N^3$. 
If the field $\k$ is infinite, $I_\mA$ is the defining ideal of the projective monomial surface $\mX_\mA$, which is the Zariski closure of the set \[\{(t_1^{a_{11}}t_2^{a_{12}}t_3^{a_{13}}:\dots:t_1^{a_{i1}}t_2^{a_{i2}}t_3^{a_{i3}} : \dots : t_1^{a_{n1}}t_2^{a_{n2}}t_3^{a_{n3}}) \in \Pn{n-1}_{\Bbbk} \mid (t_1:t_2:t_3) \in \Pn{2}_{\Bbbk} \} \subset \Pn{n-1}_{\Bbbk} \, .\]
Therefore, the coordinate ring of $\mX_\mA$, $\k[\mX_\mA] = \kx/I_\mA$ is isomorphic to the semigroup algebra $\ks$. Since $\mX_\mA \subset \Pn{n-1}_\Bbbk$ is a projective surface, the Krull dimension of $\k[\mX_\mA]$ is $3$, so we can apply the results in the previous sections. \newline

For every $s\in \N$, denote by $s\mA$ the $s$-fold sumset of $\mA$, i.e., $0\mA = \{\mathbf{0}\}$, and for $s\geq 1$, \[s\mA := \{ \bfa_{i_1}+\bfa_{i_2}+\dots+\bfa_{i_s} : 1\leq i_1\leq i_2 \leq \dots \leq i_s \leq n \} \, .\]

\begin{proposition}[{\cite[Prop.~3.3]{Colarte2022}}]
The Hilbert function of $\k[\mX_\mA]$ is given by $\hf{\k[\mX_\mA]}(s) = |s\mA|$, for all $s\in \N$.
\end{proposition}

Since $\mA$ is contained in the plane $\{(x,y,z) \in \N^3 : x+y+z = D\}$, for all $s\geq 0$,  the $s$-fold sumset of $\mA$ is also contained in a plane,
\[ s\mA \subset \{(x,y,z) \in \N^3: x+y+z = sD\} \, .\] 
For every $s\in \N$, set $H_s := \{(x,y,z) \in \N^3: x+y+z = sD\}$, and for every subset $F$ of $\mS$, set $F_s := F \cap H_s$. Moreover, if $F\subset \mS$ is a finite subset, we define the number $m(F) \in \N$ by \[m(F) := \max\{s\in \N: F_s \neq \emptyset\} \, .\]

In the case of the Apery set of $\mS$, instead of writing $(\aps)_s$, we just write $\AP_s$, and so we do with the exceptional sets. 

\begin{proposition}
For all $s\in \N$,
\[ |\AP_s| = \left( |s\mA|-3|(s-1)\mA|+3|(s-2)\mA|-|(s-3)\mA| \right) + |E_s^{3,1}| + |E_s^{2,0}| + 2|E_s^{3,0}| - |E_s^{3,3}| \, .\]
\end{proposition}

\begin{proof}
This is a direct consequence of Theorem~\ref{thm:dim3_Si} and Proposition~\ref{prop:deg_shortres}~\ref{prop:deg_shortres_b}.
\end{proof}

Since $I_\mA$ is a (standard graded) homogeneous ideal, the short resolution of $\k[\mX_\mA]$ with respect to the standard grading is 
\[\mathcal{F}: 0 \rightarrow \oplus_{\bfs\in \mS_2} A (-|\bfs|/D) \xrightarrow{\psi_2} \oplus_{\bfs\in \mS_1} A(-|\bfs|/D) \xrightarrow{\psi_1} \oplus_{\bfs\in \mS_0} A (-|\bfs|/D) \xrightarrow{\psi_0} \k[\mX_\mA] \rightarrow 0\, ,\]
and hence, the Castelnuovo-Mumford regularity of $\k[\mX_\mA]$ is
\begin{equation}
\reg{\k[\mX_\mA]} = \max \left( \left\{ \frac{|\bfs|}{D} : \bfs\in \mS_0 \right\} \cup \left\{ \frac{|\bfs|}{D}-1 : \bfs\in \mS_1 \right\} \cup \left\{ \frac{|\bfs|}{D}-2 : \bfs\in \mS_2 \right\} \right) \, ,
\label{eq:CMreg_hom}
\end{equation}
and the Hilbert series of $\k[\mX_\mA]$ is obtained from the multigraded Hilbert series by applying the transformation $t_1^{a_1}t_2^{a_2}t_3^{a_3} \mapsto t^{(a_1+a_2+a_3)/D}$.
\newline

As a direct consequence of Theorem~\ref{thm:dim3_Si} and  \eqref{eq:CMreg_hom}, we obtain the following formula for the Castelnuovo-Mumford regularity of the projective monomial surface $\mX_\mA$ in terms of the Apery and the exceptional sets of $\mS$.
{It is an analogue of the formula given in \cite[Thm.~3.7]{GG2023} for the Castelnuono-Mumford regularity of projective monomial curves.

\begin{theorem} \label{thm:CMreg_monsurf}
The Castelnuovo-Mumford regularity of $\k[\mX_\mA]$ is
    \[\reg{\k[\mX_\mA]} = \max \left\{m(\AP_\sg),m(E_\sg^{3,1})-1,m(E_\sg^{2,0})-1,m(E_\sg^{3,0})-1,m(E_\sg^{3,3})-2 \right\} \, .\]
\end{theorem}

\section{Pruning algorithm for simplicial toric rings of dimension 3} \label{sec:pruning}

Consider now, as in Section~\ref{sec:homology}, the toric ideal $I_\mA$ defined by $\mA = \{\bfa_1,\ldots,\bfa_n\} \subset \N^3$, the generating set of a simplicial semigroup $\mS$, and assume without loss of generality that the last three generators are the extremal rays of the rational cone spanned by $\mA$.
Setting $R:=\kx$ and $I=:I_\mA$, one has that $R/I$ is a simplicial toric ring of dimension 3. Moreover, for $A=:\k[x_{n-2},x_{n-1},x_n]$ and $\omega := (\omega_1,\ldots,\omega_n) \in \N^n$ with $\omega_i = |\bfa_i|$ for all~$i$, $1\leq i\leq n$, one has that $I$ is $\omega$-homogeneous and $A$ is a Noether normalization of $R/I$, so the results in Section~\ref{sec:groebner} apply.
Our aim in this section is to build the Schreyer resolution and explicitly prune it in order to build directly the short resolution of $R/I$ in this case.
\newline

Let $\mathcal G$ be the reduced Gr\"obner basis of $I$ with respect to $>_\omega$, the $\omega$-graded reverse lexicographic order. It is known that the elements in $\mathcal G$ are binomials.
Take $\mB_0$ the set of monomials not belonging to $\ini{I} + \langle x_{n-2}, x_{n-1}, x_n \rangle.$ Consider  $\chi: R\rightarrow R$ the evaluation morphism defined by $\chi(x_i) = x_i$ for $i \in \{1,\ldots,n-3\}$ and $\chi(x_j) = 1$ for $j \in \{n-2, n-1,n\}$, and set $J$ the extension of $\ini{I}$ by $\chi$. Now, for every $u \in \mB_0 \cap J$,  $G(I_u)$ denotes the minimal monomial generating set of 
\[I_u := \left( \ini{I} : u \right) \cap \k[x_{n-2},x_{n-1},x_n]  \, .\]
Since the generators of $\ini{I}$ do not involve the variable $x_n$ because the ideal $I$ is prime and $>_\omega$ is a reverse lexicographic order, every element in $G(I_u)$ is a monomial of the form $x_{n-2}^a x_{n-1}^b$ with $a,b \in \N$. Denote $\ell_u := |G(I_u)|$ and write $G(I_u) = \{M_{(u,1)},\ldots,M_{(u,{\ell_u})}\}$, where the elements of $G(I_u)$ are sorted lexicographically, i.e., $M_{(u,1)} >  \cdots >  M_{(u,\ell_u)}$ with respect to the lexicographic order $x_n > x_{n-1} > x_{n-2}$. 
Now consider the set of monomials \[ \mB_1' = \{u M_{(u,i)}  \mid u \in \mB_0 \cap J , 1 \leq i \leq \ell_u \}.\] 
For each $\bx^{\alpha} = u  M_{(u,i)} \in \mB_1'$,  where $u \in \mB_0 \cap J$ and $M_{(u,i)} = x_{n-2}^a x_{n-1}^b \in G(I_u)$, we take $r_\alpha$ the remainder of the division of $\bx^{\alpha}$ by $\mG$. Since $\mG$ consists of binomials and $M_{(u,i)} \in G(I_u)$, then $r_\alpha = x_{n-2}^{a'} x_{n-1}^{b'} x_n^{c'} v$ for some $a',b',c' \in \N$ such that $\gcd(M_{(u,i)}, x_{n-2}^{a'} x_{n-1}^{b'}) = 1$ and some $v \in \mB_0$.  
By Theorem~\ref{them:GBschreyer}, the set \[ \mH = \{\bfh_{(u,i)} := M_{(u,i)} \cdot \bfeps_u - x_{n-2}^{a'} x_{n-1}^{b'} x_{n-1}^{c'} \cdot \bfeps_v \mid u \in \mB_0 \cap J , 1 \leq i \leq \ell_u \}\] 
is the reduced Gr\"obner basis for the Schreyer-like monomial order $>_{\rm SL}$ in Definition~\ref{def:SLorder}, and $\ini{\bfh_{(u,i)}} = M_{(u,i)} \cdot \bfeps_u$
by Remark~\ref{rem:initial}. Applying Schreyer's Theorem, one gets that the syzygies of $\mathcal H$ are obtained by reducing the $S$-polynomials of all pairs of elements in $\mH$ by $\mH$. Note that only $S$-polynomials of the form 
$S(\bfh_{(u,i)}, \bfh_{(u,j)})$ with $u \in \mB_0 \cap J$ and $1 \leq i < j \leq \ell_u$ must be considered and reduced since the other $S$-polynomials are zero. Furthermore, since the monomials $M_{(u,i)}$ only involve variables $x_{n-2}$ and $x_{n-1}$ and have been lexicographically sorted, $M_{(u,1)}>\ldots>M_{(u,\ell_u)}$, we only need to consider the reductions of the $S$-polynomials $S(\bfh_{(u,i)}, \bfh_{(u,i+1)})$ with $u \in \mB_0 \cap J$ and $1 \leq i < \ell_u$ since the other ones will be discarded when the resulting Gr\"obner basis of the syzygy module is made minimal. 
This implies that the initial terms of the resulting syzygies are pure powers of $x_{n-2}$ located in different copies of $A$, and hence the module of syzygies of $\mH$ obtained by applying Schreyer's Theorem is free. The Schreyer resolution of $R/I$ has thus at most two steps, and it shows as follows:
{\small
\begin{equation}
0 \rightarrow \oplus_{v\in \mB_2'} A(-\degw (v)) \xrightarrow{\psi_2'}  \oplus_{v\in \mB_1'} A(-\degw (v)) \xrightarrow{\psi_1'} \oplus_{v\in \mB_0} A(-\degw (v)) \xrightarrow{\psi_0} R/I \rightarrow 0 \, .
\label{eq:Schreyer_dim3}
\end{equation}}

Algorithm~\ref{alg:dim3} below takes advantage of the previous discussion and builds $\mB_0$, $\mB_1'$, and $\mB_2'$, the sets of monomials in $R$ involved in the above resolution.
It is worth pointing out that this algorithm involves only a Gr\"obner basis computation and Gr\"obner-free manipulations with monomial ideals. It has been implemented in the function {\tt schreyerResDim3} of \cite{github_shortres}.

\begin{algorithm} 
\caption{Computation of the sets $\mB_i'$ for a simplicial toric ring of dim. $3$.}
\label{alg:dim3}
\begin{flushleft}
    \textbf{Input:} $I \subset R = \kx$ a simplicial toric ideal of dimension $3$ with variables in Noether position. \\
    \textbf{Output:} The sets of monomials $\mB_0,\mB_1',\mB_2' \subset R$ involved in the Schreyer resolution \eqref{eq:Schreyer_dim3} of $R/I$ as $A$-module, $A = \k[x_{n-2},x_{n-1},x_n]$.
\end{flushleft}
\begin{algorithmic}[1]
\State $\mB_0$ $\gets$ monomial $\k$-basis of $R / \ini{I}+\id{x_{n-2},x_{n-1},x_n}$ for the degrevlex order $>_\omega$.
\State $J \gets \chi \left( \ini{I} \right) . R$, where $\chi: R \rightarrow R$ is defined by $\chi(x_i) = x_i$ for $i \in \{1,\ldots,n-3\}$, and $\chi(x_{n-2}) =\chi(x_{n-1})=\chi(x_{n})= 1$.
\State $I_u \gets \left( \ini{I} : u \right) \cap A$, $\forall u \in \mB_0 \cap J$.
\State $G(I_u) \gets$ minimal generating set of $I_u$, $\forall u \in \mB_0 \cap J$; $G(I_u) = \{M_{(u,1)},\ldots,M_{(u, \ell_u)}\}$ ordered lexicographically with $x_n > x_{n-1} > x_{n-2}$.
\State $\mB_1' \gets \{ u \cdot M_{(u,i)} \mid u \in \mB_0 \cap J, 1 \leq i \leq \ell_u\}$.
\State $L_u \gets \{ \lcm(M_{(u,i)},M_{(u,i+1)}) \mid 1 \leq i < \ell_u\}, \ \forall u \in \mB_0 \cap J$ such that $\ell_u \geq 2$.
\State $\mB_2' \gets \{ u \cdot M \mid u\in \mB_0 \cap J , \ell_u \geq 2, \text{ and } M \in L_u \}$.
\end{algorithmic}
\end{algorithm}

As Examples~\ref{ex:noprincipalnominimal} and ~\ref{ex:toric_nomin} show, even when $R/I$ is a $3$-dimensional simplicial toric ring, the resolution \eqref{eq:Schreyer_dim3} might not be minimal.

\begin{example} \label{ex:toric_nomin}
Set $R := \Q[x_1,\ldots,x_6]$, and let $I$ be the toric ideal determined by $\mA = \{(7,2,3),(1,8,3),(3,8,1),(12,0,0),(0,12,0),(0,0,12)\}$.
One has that $I$ is a homogeneous toric ideal and $A = \Q[x_4,x_5,x_6]$ is a Noether normalization of $R/I$, hence $R/I$ is a 3-dimensional simplicial toric ring. Applying Algorithm~\ref{alg:dim3} we obtain that $|\mB_0| = 204$, $|\mB_1'| = 174$ and $|\mB_2'| = 42$. However, the Betti diagram of the short resolution, obtained by using the function {\tt shortRes} of \cite{github_shortres}, is the following:
\newpage

\begin{verbatim}
           0     1     2
------------------------
    0:     1     -     -
    1:     3     -     -
    2:     6     1     -
    3:    10     3     -
    4:    15     6     -
    5:    21    10     -
    6:    26    15     -
    7:    29    20     -
    8:    32    26     1
    9:    29    26     2
   10:    20    19     2
   11:     9     9     1
   12:     2     2     -
   13:     1     1     -
------------------------
total:   204   138     6
\end{verbatim}
\end{example}

Our next aim is thus to minimalize Schreyer's resolution~\eqref{eq:Schreyer_dim3} using the results from Section~\ref{sec:homology}. We will show how to obtain subsets $\mB_1 \subset \mB_1'$ and $\mB_2 \subset \mB_2'$, such that $\mB_1$ and $\mB_2$ provide the actual shifts that appear in the short resolution of $R/I$. We will refer to this process as {\it pruning} the resolution. Note that, by Proposition~\ref{prop:deg_shortres}~\ref{prop:deg_shortres_a},
\[e(R/I) = |\mB_0| - |\mB_1| + |\mB_2| = |\mB_0| - |\mB_1'| + |\mB_2'|\] and, in particular, $|\mB_1'\setminus \mB_1|= |\mB_2' \setminus \mB_2|$. \newline

In the process of pruning the resolution, we will use the following result several times.

\begin{proposition}  
\label{prop:pertencesemig}Let $\mS = \langle \bfa_1,\ldots,\bfa_n \rangle \subset \N^d$ be an affine semigroup and $\bfb, \bfc \in \mS$. Write $\bfb = \sum_{i = 1}^n \beta_i \bfa_i$ and $\bfc = \sum_{i = 1}^n \gamma_i \bfa_i$ with $\beta_i, \gamma_i \in \N$ and consider the monomials $\mathbf{x}^{\beta} := x_1^{\beta_{1}} \cdots x_n^{\beta_{n}}$ and $\mathbf{x}^{\gamma} := x_1^{\gamma_{1}} \cdots x_n^{\gamma_{n}} \in \k[x_1,\ldots,x_n]$. Then, $\bfb - \bfc \in \mS$ if and only if $\bx^{\beta} \in I_\mA + \langle \bx^{\gamma} 
 \rangle$. 
\end{proposition}

\begin{proof}
We know that $R / I_\mA$ and $\k[\mS]$ are isomorphic as graded $\k$-algebras, and denote by $\tilde{\varphi}$ the corresponding graded isomorphism. Now, consider the ideal $\langle \mathbf{t}^{\bfc} \rangle$ of $\k[\mS]$, and the canonical projection map $\pi: \k[\mS] \rightarrow \k[\mS]/\langle \bft^\bfc\rangle$.
Since $\tilde{\varphi}(\mathbf{x}^{\gamma})= \mathbf{t}^{\bfc}$, we have that  $\ker(\pi \circ \tilde{\varphi}) = (I_\mA + \langle \mathbf{x}^{\gamma}\rangle) / I_\mA$. Thus, by the third isomorphism theorem, there is a graded isomorphism of $\k$-algebras
\[
\Psi: \mathbb \k[\mathbf{x}] / (I_\mA + \langle \mathbf{x}^{\gamma}\rangle) \longrightarrow \mathbb \k[\mS]/ \langle \mathbf{t}^{\bfc} \rangle .
\] 

Moreover,  $\k[\mS]/ \langle \mathbf{t}^{\bf{c}} \rangle$ has a unique monomial basis, which is $\{\mathbf{t}^{\bf{d}} \, \vert \, \bf{d} \in \mS\ {\rm and }\ \bf{d} - \bf{c} \notin \mS\}$. Finally, observe that the image of a monomial by $\Psi$ is a monomial, and hence \[ \bx^{\beta}\in   I_\mA + \langle \bx^{\gamma} \rangle \Longleftrightarrow  \Psi(\bx^{\beta}) = 0 \Longleftrightarrow \bfb - \bfc \in \mS\]
and we are done.
\end{proof}

To achieve our goal, consider the subset $C \subset \mB_1'$ defined by $C = \{v\cdot x_{n-1}^b \in \mB_1' \mid v\in \mB_0 \text{ and } b\geq 2\}$. The following result shows that the elements in $\mB_1' \setminus \mB_1$ belong to $C$.

\begin{lemma} \label{lemma:C}
$\mB_1' \setminus \mB_1 \subset C$. 
\end{lemma}
\begin{proof}
Consider $\bx^\alpha \in \mB_1' \setminus \mB_1$ and denote by $\bfh_\alpha$ the corresponding element of $\mH$.
Since $\bx^\alpha \notin \mB_1$, there exist  $u \in \mB_0 \cap J$ and $1 \leq i < \ell_u$ such that there appears a nonzero constant multiplying $\bfh_\alpha$ in the reduction of $S(\bfh_{(u,i)},\bfh_{(u,i+1)})$ by $\mH$.
If $\bfh_{(u,i)} = x_{n-2}^{a_1}x_{n-1}^{b_1} \bfeps_u-x_{n}^{c_1} \bfeps_v$ and $\bfh_{(u,i+1)} = x_{n-2}^{a_2}x_{n-1}^{b_2} \bfeps_u -x_{n}^{c_2} \bfeps_w$, for some $v,w\in \mB_0$, $a_i,b_i,c_i \in \N$ ($i=1,2$) with $c_1,c_2\geq 1$, $a_1<a_2$, and $b_1>b_2$, then 
\[S(\bfh_{(u,i)},\bfh_{(u,i+1)}) = x_{n-2}^{a_2-a_1} \bfh_{(u,i)} - x_{n-1}^{b_1-b_2} \bfh_{(u,i+1)} = x_n \left( x_{n-1}^{b_1-b_2} x_n^{c_2-1} \bfeps_w - x_{n-2}^{a_2-a_1} x_n^{c_1-1} \bfeps_v \right) \, . \]
Hence, the reduction of $S(\bfh_{(u,i)},\bfh_{(u,i+1)})$ by $\mH$ does not involve nonzero constants.

Therefore, $\bfh_{(u,i)} = x_{n-2}^{a_1} \bfeps_u - x_{n-1}^{b_1} \bfeps_v$, and $\bfh_{(u,i+1)} = x_{n-2}^{a_2}x_{n-1}^{b_2} \bfeps_u -x_n^c \bfeps_w$, for some $v,w\in \mB_0$ and $a_1,a_2,b_1,b_2,c \in \N$ with $a_1,b_1,b_2,c \geq 1$, $a_2 <a_1$ and $b_1+b_2=b \geq 2$. Hence, 
\[S(\bfh_{(u,i)},\bfh_{(u,i+1)}) = x_{n-1}^{b_2} \bfh_{(u,i)} -x_{n-2}^{a-a'} \bfh_{(u,i+1)} = - x_{n-1}^{b} \bfeps_v + x_n \left( x_{n-2}^{a-a'}x_n^{c-1} \bfeps_w \right) \, ,\]
and since there appears a nonzero constant in the reduction of $S(\bfh_{(u,i)},\bfh_{(u,i+1)})$, one has that $x_{n-1}^{b} \in G(I_v)$, where $I_v = \left( \ini{I} : v \right) \cap A$. Thus, $\bx^\alpha = vx_{n-1}^{b}\in~C$.
\end{proof}

\begin{remark}
As a direct consequence of the previous result, if $C = \emptyset$, then $\mB_1' = \mB_1$, $\mB_2' = \mB_2$, and hence the Schreyer resolution \eqref{eq:Schreyer_dim3} is already minimal.
\end{remark}

The inclusion $\mB_1' \setminus \mB_1 \subset C$ can be strict or not. In fact, if $C \neq \emptyset$, both cases $\mB_1' = \mB_1$ and $\mB_1'\setminus\mB_1 = C$ can happen, as the following examples show.

\begin{example}
In this example, computations are performed over the field $\Q$.
\begin{enumerate}
\item Set $\mA := \{(1,0,3),(3,0,1),(0,1,3),(3,1,0),(0,3,1),(1,3,0),(4,0,0),(0,4,0)$, $(0,0,4)\}$, and let $I$ be the toric ideal determined by $\mA$. Applying Algorithm~\ref{alg:dim3}, one gets that $|\mB_0| = 28$, $|\mB_1'| = 18$, and $|\mB_2'| = 6$. In this case, $\mB_1 = \mB_1'$ although $|C| = 3$ since the Betti diagram of the short resolution given by Algorithm~\ref{alg:short_res} is
\begin{verbatim}
           0     1     2
------------------------
    0:     1     -     -
    1:     6     -     -
    2:    12     3     -
    3:     6     6     -
    4:     3     9     6
------------------------
total:    28    18     6
\end{verbatim}
    
\item If $\mA$ is the set in Example~\ref{ex:toric_nomin}, $|\mB_1'| = 174$, $|\mB_1| = 138$ and $|C| = 36$, so $\mB_1'\setminus\mB_1 = C$.
\end{enumerate}
\end{example}

For each $\bx^\beta \in C$, denote by $r_\beta$ the remainder of $\bx^\beta$ by the reduced Gröbner basis of $I$ for the $\omega$-graded reverse lexicographic order $>_\omega$ in $R$. Since $x_{n-1}$ divides $\bx^\beta$, then $r_\beta$ is a multiple of $x_n$.
Consider the partition $C = C_1 \sqcup C_2$, where
\[\begin{split}
C_1 &= \{ \bx^\beta \in C \mid r_\beta = wx_{n-2}^a x_n^c, \text{ for some } a,c\geq 1, w \in \mB_0 \}\, , \text{ and} \\
C_2 &= \{ \bx^\beta \in C \mid r_\beta = w x_n^c, \text{ for some } c\geq 1, w \in \mB_0 \} \, .
\end{split}\]

We now show that one can decide whether a monomial $\bx^\beta \in C$ is in $\mB_1$ or not just by looking at its $\mS$-degree. More precisely, it suffices to check if $\degs{\bx^\beta} = \sum_{i=1}^n \beta_i \bfa_i$ appears as a shift in the first step of the short resolution, and this happens if and only if $\degs{\bx^\beta} \in E_\mS^{3,1} \cup E_\mS^{2,0} \cup E_\mS^{3,0}$ by Theorem~\ref{thm:dim3_Si}. In Theorem~\ref{thm:prunB1}, we characterize when the latter holds in terms of some monomials that may belong to the ideal
$I_\mA + \langle x_{n-2} \rangle$ or not.
We will use the following easy lemma.
As in Section \ref{sec:homology}, set $\bfe_i := \bfa_{n-3+i}$ for all $i\in \{1,2,3\}$ and $\mE := \{\bfe_1,\bfe_2,\bfe_3\}$.

\begin{lemma} \label{lemma:6}
Let $\bx^\beta = vx_{n-1}^b \in C$ and set $\bfs = \degs{\bx^\beta}$.
\begin{enumerate}[(1)]
    \item\label{lemma:6_1} If $\bx^\beta \in C_1$, then \[\bfs-\bfe_i \in \mS, \forall i=1,2,3; \, \bfs-(\bfe_1+\bfe_3) \in \mS; \, \text{and } \,\bfs-(\bfe_2+\bfe_3) \notin \mS \, .\]
    \item\label{lemma:6_2} If $\bx^\beta \in C_2$, then \[\bfs-\bfe_2 \in \mS; \, \bfs-\bfe_3 \in \mS; \, \text{and } \, \bfs-(\bfe_2+\bfe_3) \notin \mS \, .\]
\end{enumerate}
\end{lemma}

\begin{figure}[htbp]
\centering
\begin{subfigure}[b]{0.45\linewidth}
\centering
\begin{tikzpicture}[scale=1.5]
\node at (1.25,1.1,0) {$\bfs$};
  \draw[thin](1,1,0)--(0,1,0)--(0,1,1)--(1,1,1)--(1,1,0)--(1,0,0)--(1,0,1)--(0,0,1)--(0,1,1);
  \draw[thin](1,1,1)--(1,0,1);
  \draw[thin,dashed](1,0,0)--(0,0,0)--(0,1,0);
  \draw[thin,dashed](0,0,0)--(0,0,1);
  \draw [red] plot [only marks, mark size=1.2, mark=square*, mark options = {fill=white}] coordinates {(0,0,1) (1,0,1)};
  \draw [blue] plot [only marks, mark size=1.2, mark=*] coordinates {(1,0,0) (0,1,0) (1,1,0) (0,1,1) (1,1,1)};
\end{tikzpicture}
\caption{Situation in Lemma~\ref{lemma:6}~\ref{lemma:6_1}.}
\end{subfigure}
\begin{subfigure}[b]{0.45\linewidth}
\centering
\begin{tikzpicture}[scale=1.5]
\node at (1.25,1.1,0) {$\bfs$};
  \draw[thin](1,1,0)--(0,1,0)--(0,1,1)--(1,1,1)--(1,1,0)--(1,0,0)--(1,0,1)--(0,0,1)--(0,1,1);
  \draw[thin](1,1,1)--(1,0,1);
  \draw[thin,dashed](1,0,0)--(0,0,0)--(0,1,0);
  \draw[thin,dashed](0,0,0)--(0,0,1);
  \draw [red] plot [only marks, mark size=1.2, mark=square*, mark options = {fill=white}] coordinates {(0,0,1) (1,0,1)};
  \draw [blue] plot [only marks, mark size=1.2, mark=*] coordinates {(1,0,0) (1,1,0) (1,1,1)};
\end{tikzpicture}
\caption{Situation in Lemma~\ref{lemma:6}~\ref{lemma:6_2}.}
\end{subfigure}
\caption{}
\label{fig:lemma6}
\end{figure}

\begin{proof}
Let us prove~\ref{lemma:6_1}. If $\bx^\beta = vx_{n-1}^b \in C_1$, there exist a monomial $w\in \mB_0$ and natural numbers $a,c\geq 1$ such that $vx_{n-1}^b-wx_{n-2}^a x_n^c \in I_\mA$. From this fact, it follows that $\bfs-\bfe_i \in \mS$ for $i=1,2,3$, and $\bfs-(\bfe_1+\bfe_3) \in \mS$. Suppose by contradiction that $\bfs-(\bfe_2+\bfe_3) \in \mS$. Then, there exists a monomial $M \in \k[x_1,\ldots,x_n]$ such that $vx_{n-1}^b- M x_{n-1}x_n \in I_\mA$. Since $I_\mA$ is prime, then $vx_{n-1}^{b-1}- M x_n \in I_\mA$, and hence $vx_{n-1}^{b-1} \in \ini{I_\mA}$, which contradicts with the minimality of $x_{n-1}^b \in G(I_v)$. The proof of~\ref{lemma:6_2} is analogous.
\end{proof}

\begin{theorem}\label{thm:prunB1}
Let $\bx^\beta = vx_{n-1}^b \in C$. 
\begin{enumerate}[(1)]
    \item\label{thm:prunB1_1} If $\bx^\beta \in C_1$, then
    \[vx_{n-1}^b \in \mB_1 \Longleftrightarrow vx_{n-1}^{b-1} \notin I_\mA + \id{x_{n-2}} \, .\]
    \item\label{thm:prunB1_2} If $\bx^\beta \in C_2$, denote by $wx_n^c$ the remainder of $\bx^\beta$ by $\mG$. Then,
    \[vx_{n-1}^b \in \mB_1 \Longleftrightarrow vx_{n-1}^{b-1} \notin I_\mA + \id{x_{n-2}} \text{ or } wx_n^{c-1} \notin I_\mA + \id{x_{n-2}} \, .\]
\end{enumerate}
Therefore,
\begin{multline*}
\mB_1 = \left( \mB_1' \setminus C \right) \cup \{ vx_{n-1}^b \in C_1 \mid vx_{n-1}^{b-1} \notin I_\mA + \id{x_{n-2}}\} \\
\cup \{vx_{n-1}^b \in C_2 \mid vx_{n-1}^{b-1} \notin I_\mA + \id{x_{n-2}} \text{ or } wx_n^{c-1} \notin I_\mA+\id{x_{n-2}}\} \, .
\end{multline*}

\end{theorem}

\begin{proof}
By Theorem~\ref{thm:dim3_Si}, we know that the multiset of $\mS$-degrees appearing in the first step of the short resolution is \[ \mS_1 = E_\mS^{3,1} \cup E_\mS^{2,0} \cup E_\mS^{3,0} \cup E_\mS^{3,0}; \]
we observe that in $\mS_1$ the elements of $E_\mS^{3,1} \cup E_\mS^{2,0}$ have multiplicity $1$, and the elements of $E_\mS^{3,0}$ have multiplicity two.
We know that $\mS_1$ is a (multi)subset of  \[ \mS_1' := \{ \degs{\bx^\alpha} \, \vert \, \bx^{\alpha} \in \mB_1' \setminus C\}  \cup \{ \degs{\bx^\alpha} \, \vert \, \bx^{\alpha} \in  C \}. \]

\noindent {\it Claim:} Whenever $\bfs \in \mS_1$, its multiplicities in $\mS_1$ and in $\mS_1'$ coincide.

\noindent {\it Proof of the claim:} By Lemma~\ref{lemma:C}, we know that $\{ \degs{\bx^\alpha} \, \vert \, \bx^{\alpha} \in \mB_1' \setminus C \}$ is a (multi)subset of $\mS_1$. Hence, to derive the claim it suffices to prove that: \begin{enumerate}[(i)]
\item\label{thm:prunB1_claim_i} distinct elements of $C$ have distinct $\mS$-degrees, and
\item\label{thm:prunB1_claim_ii} whenever an element of $\mB_1' \setminus C$ and an element of $C$ have the same $\mS$-degree, then this $\mS$-degree belongs to $E_\mS^{3,0}$ and has multiplicity exactly two in $\mS_1'$.
\end{enumerate}

To prove \ref{thm:prunB1_claim_i}, consider two elements in $C$ with the same $\mS$-degree, namely, $\bx^\alpha = ux_{n-1}^b$ and $\bx^\beta = u' x_{n-1}^{b'}$ and assume that $b \geq b'$. Then it follows that $f = ux_{n-1}^{b-b'} - u' \in I_\mA$, so $f = 0$, and hence $u=u'$ and $b=b'$.

To prove \ref{thm:prunB1_claim_ii}, consider $\bx^\alpha \in \mB_1' \setminus C$ and $\bx^\beta \in C$ with $\bfs:= \degs{\bx^\alpha} = \degs{\bx^\beta}$. We write $\bx^\beta = ux_{n-1}^b$, $b\geq 2$, $\bx^\alpha = u'x_{n-2}^{a'}x_{n-1}^{b'}$, $a'+b'\geq 1$.
Suppose first that $\bx^\beta \in C_1$, i.e., $r_\beta = vx_{n-2}^{a}x_n^{c}$ for some $a,c \in \Z^+$. If $a'\geq 1$, then $u'x_{n-2}^{a'-1}x_{n-1}^{b'}-vx_{n-2}^{a-1}x_n^c \in I_\mA$, so $u'x_{n-2}^{a'-1}x_{n-1}^{b'} \in \ini{I_\mA}$, and hence $x_{n-2}^{a'-1}x_{n-1}^{b'} \in I_{u'}$, contradicting the minimality of $x_{n-2}^{a'}x_{n-1}^{b'} \in G(I_{u'})$. Therefore, $a'=0$, so $\degs{ux_{n-1}^b} = \degs{u'x_{n-1}^{b'}}$, which implies that $u=u'$ and $b=b'$, a contradiction. Hence, $\bx^\beta \in C_2$, i.e., $r_\beta = vx_n^c$, for some $c\in \Z^+$. Now, let us see that $\bfs \in E_\mS^{3,0}$. If $b' \geq 1$, then $ux_{n-1}^{b-1}-u'x_{n-2}^{a'}x_{n-1}^{b'-1} \in I_\mA$ is a nonzero binomial and neither $ux_{n-1}^{b-1}$ nor $u'x_{n-2}^{a'}x_{n-1}^{b'-1}$ belongs to $\ini{I_\mA}$, which is impossible. This proves that $b'=0$. Since, $\bfs = \degs{u'x_{n-2}^{a'}} = \degs{ux_{n-1}^b} = \degs{vx_n^c} \in \mS_1$, then either $\bfs \in E_\mS^{3,1}$ or $\bfs \in E_\mS^{3,0}$. 
Suppose $\bfs \in E_\mS^{3,1}$. Then there exists $w\in \mB_0$, $a'',b'', c'' \in \N$ with at least two of them nonzero, such that $\bfs = \degs{wx_{n-2}^{a''}x_{n-1}^{b''}x_n^{c''}}$. Combining this with $\bfs = \degs{u'x_{n-2}^{a'}}$ (if $a'' \neq 0$) or $\bfs = \degs{ux_{n-1}^b}$ (if $b'' \neq 0$), we get a contradiction. Hence, $\bfs \in E_\mS^{3,0}$. 
Finally, let us see that there does not exist $\bx^\gamma \in \mB_1' \setminus C$, $\bx^\gamma \neq \bx^\alpha$, such that $\degs{\bx^\gamma} = \bfs$. Let $\bx^\gamma = u''x_{n-2}^{a''}x_{n-1}^{b''} \in \mB_1' \setminus C$, $\bx^\gamma \neq \bx^\alpha$, such that $\degs{\bx^\gamma} = \bfs$. Then, $a'',b''\in \Z^+$, or $a''\in \Z^+$ and $b''=0$, or $a''=0$ and $b''=1$. Proceeding as before each of these three cases leads to a contradiction. Therefore, the claim is proved.

As a consequence of the {\it Claim}, one has a criterion to detect if an element of $\mB_1'$ belongs to $\mB_1$ or not. More precisely, let $\bx^\alpha \in \mB_1'$, then:
\[ \bx^\alpha \in \mB_1  \Longleftrightarrow \degs{\bx^\alpha} \in \mS_1 \, .\]  
 We now use this criterion to prove~\ref{thm:prunB1_1} and~\ref{thm:prunB1_2}.

Let $\bx^\beta = vx_{n-1}^b \in C_1$ and set $\bfs = \degs{\bx^\beta}$. By Lemma~\ref{lemma:6}~\ref{lemma:6_1}, one has that $\bx^\beta \in \mB_1$ \iff{} $\bfs-(\bfe_1+\bfe_2) \notin \mS$. Then, by Proposition~\ref{prop:pertencesemig}, one has that $\bfs-(\bfe_1+\bfe_2) \in \mS$ \iff{} $vx_{n-1}^{b-1} \in I_\mA + \id{x_{n-2}}$.

Let $\bx^\beta = vx_{n-1}^b \in C_2$ and set $\bfs = \degs{\bx^\beta}$ and $r_\beta = wx_n^c$ the remainder of $\bx^\beta$ by the reduced Gröbner basis of $I_\mA$. By Lemma~\ref{lemma:6}~\ref{lemma:6_2}, one has that $\bx^\beta \in \mB_1$ \iff{} $\bfs-(\bfe_1+\bfe_2)  \notin \mS$ or $\bfs-(\bfe_1+\bfe_3) \notin \mS$. Hence, the result follows again from Proposition~\ref{prop:pertencesemig}.

The last claim in the theorem is a direct consequence of~\ref{thm:prunB1_1} and~\ref{thm:prunB1_2}. 
\end{proof}

In Theorem~\ref{thm:prunB1}, we have obtained a test to decide algebraically if a monomial $\bx^\beta \in C \subset \mB_1'$ is in $\mB_1$ or not, and hence we can obtain the set $\mB_1$. To apply this criterion, one only has to test the membership of some monomials to the ideal $I_\mA + \id{x_{n-2}}$. 
Now, we do something similar to obtain the set $\mB_2 \subset \mB_2'$. 

\begin{lemma} \label{lemma:8}
Let $\bx^\alpha \in \mB_2' \setminus \mB_2$, and set $\bfs = \degs{\bx^\alpha}$, then
\[\bfs-\bfe_i \in \mS, \forall i=1,2,3; \, \bfs-(\bfe_1+\bfe_2) \in \mS; \, \bfs-(\bfe_1+\bfe_3) \in \mS; \text{and } \,\bfs-(\bfe_1 + \bfe_2+\bfe_3) \notin \mS \, .\]
\end{lemma}

\begin{proof}
If a monomial $\bx^\alpha \in \mB_2'$ is not in $\mB_2$, then it comes from a $S$-polynomial $S(\bfh,\bfh')$, $\bfh,\bfh'\in \mH$, such that there appears a nonzero constant in the reduction of $S(\bfh,\bfh')$ by the Gröbner basis $\{\bfh_\alpha \mid \bx^\alpha \in \mB_1'\}$. The syzygies $\bfh,\bfh'$ have expressions $\bfh = x_{n-2}^a \bfeps_u- x_{n-1}^b \bfeps_v$, and $\bfh' = x_{n-2}^{a'}x_{n-1}^{b'} \bfeps_u -x_n^c \bfeps_w$, for some $u,v,w\in \mB_0$,  $a,b,b',c \geq 1$ and $a' \in \N$, with $a'<a$, as in the proof of Lemma~\ref{lemma:C}. Therefore, $\bx^\alpha = u \cdot \lcm ( x_{n-2}^a,x_{n-2}^{a'}x_{n-1}^{b'} ) = u \cdot x_{n-2}^a x_{n-1}^{b'}$, and we note that \[\degs{u  x_{n-2}^a x_{n-1}^{b'}} = \degs{vx_{n-1}^{b+b'}} = \degs{wx_{n-2}^{a-a'}x_n^c} \, .\]
From the previous equalities, we deduce that $\bfs-\bfe_i \in \mS$ for all $i=1,2,3$, $\bfs-(\bfe_1+\bfe_2) \in \mS$, and $\bfs-(\bfe_1+\bfe_3) \in \mS$. Let us prove that $\bfs-(\bfe_1 + \bfe_2+\bfe_3) \notin \mS$. Assume by contradiction that $\bfs-(\bfe_1 + \bfe_2+\bfe_3) \in \mS$. Then, there exist a monomial $M \in \k[x_1,\ldots,x_n]$, such that $ux_{n-2}^a x_{n-1}^{b'}-x_{n-2}x_{n-1}x_n M \in I_\mA$, so $ux_{n-2}^{a-1} x_{n-1}^{b'-1}- x_n M \in I_\mA$. 
Therefore, there exist natural numbers $a'' \leq a-1$ and $b'' \leq b'-1$ such that $x_{n-2}^{a''} x_{n-1}^{b''} \in G(I_u)$. Since $x_{n-2}^{a'}x_{n-1}^{b'} \in G(I_u)$ and $b'' < b'$, then $a'' >a'$. 
Hence, $\lcm \left( x_{n-2}^{a''} x_{n-1}^{b''} , x_{n-2}^{a'} x_{n-1}^{b'} \right) = x_{n-2}^{a''}x_{n-1}^{b'}$, which is a proper divisor of $x_{n-2}^a x_{n-1}^{b'}$, a contradiction with $\bx^\alpha = u x_{n-2}^a x_{n-1}^{b'}\in \mB_2'$.
\end{proof}

\begin{figure}[htbp]
\centering
\begin{tikzpicture}[scale=1.5]
\node at (1.25,1.1,0) {$\bfs$};
  \draw[thin](1,1,0)--(0,1,0)--(0,1,1)--(1,1,1)--(1,1,0)--(1,0,0)--(1,0,1)--(0,0,1)--(0,1,1);
  \draw[thin](1,1,1)--(1,0,1);
  \draw[thin,dashed](1,0,0)--(0,0,0)--(0,1,0);
  \draw[thin,dashed](0,0,0)--(0,0,1);
  \draw [red] plot [only marks, mark size=1.2, mark=square*, mark options = {fill=white}] coordinates {(0,0,1)};
  \draw [blue] plot [only marks, mark size=1.2, mark=*] coordinates {(0,0,0) (1,0,0) (0,1,0) (1,1,0) (0,1,1) (1,1,1)};
\end{tikzpicture}
\caption{Situation in Lemma~\ref{lemma:8}.}
\label{fig:lemma8}
\end{figure}
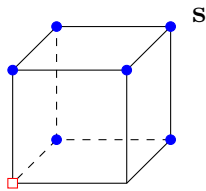

\begin{theorem} \label{thm:prunB2}
For all $\bx^\alpha = u x_{n-2}^a x_{n-1}^b \in \mB_2'$,
\[\bx^\alpha \in \mB_2 \Longleftrightarrow u x_{n-2}^a x_{n-1}^{b-1} \in I_\mA + \id{x_n} \, .\]
Therefore,
\[\mB_2 = \{ u x_{n-2}^{a} x_{n-1}^{b} \in \mB_2' \mid u x_{n-2}^a x_{n-1}^{b-1} \in I_\mA + \id{x_n} \} \, .\]
\end{theorem}

\begin{proof}
By Theorem~\ref{thm:dim3_Si}, we know that the multiset of $\mS$-degrees appearing in the second step of the short resolution is $\mS_2 = E_\mS^{3,3},$
and every element appears with mutiplicity $1$. We know that $\mS_2$ is a (multi)subset of  \[ \mS_2' := \{ \degs{\bx^\alpha} \, \vert \, \bx^{\alpha} \in \mB_2'\}. \]

\noindent {\it Claim:} Distinct elements of $\mB_2'$ have distinct $\mS$-degrees.

\noindent{\it Proof of the claim.} 
Take $\bx^\alpha = u x_{n-2}^{a}x_{n-1}^{b} \in \mB_2'$ and $\bx^\beta = u' x_{n-2}^{a'}x_{n-1}^{b'} \in \mB_2'$, $a,a',b,b' \in \Z^+$, such that $\degs{\bx^\alpha} = \degs{\bx^\beta}$. Then, $u x_{n-2}^{a}x_{n-1}^{b} - u' x_{n-2}^{a'}x_{n-1}^{b'} \in I_\mA$. If $u=u'$, then $a=a'$ and $b=b'$, and hence $\bx^\alpha = \bx^\beta$. Now suppose $u\neq u'$. In this case, $ux_{n-2}^{a}x_{n-1}^{b-1}-u'x_{n-2}^{a'}x_{n-1}^{b'-1} \in I_\mA$. Assume without loss of generality that its initial term is $ux_{n-2}^{a}x_{n-1}^{b-1} \in \ini{I_\mA}$, so $x_{n-2}^{a} x_{n-1}^{b-1} \in I_u$, contradicting the minimality of $x_{n-2}^{a}x_{n-1}^b \in G(I_u)$. Hence, the {\it Claim} follows.

As a consequence of the {\it Claim}, one has a criterion the detect if an element of $\mB_2'$ belongs to $\mB_2$ or not. More precisely, let $\bx^\alpha \in \mB_2'$, then:
\[ \bx^\alpha \in \mB_2  \Longleftrightarrow \degs{\bx^\alpha} \in \mS_2 = E_\mS^{3,3} \]  

By Lemma~\ref{lemma:8}, one has that for all $\bx^\alpha \in \mB_2'$,
\[\degs{\bx^\alpha} \in E_\mS^{3,3} \Longleftrightarrow \degs{\bx^\alpha}-(\bfe_2+\bfe_3) \in \mS \, . \]
Therefore, the result follows from Proposition~\ref{prop:pertencesemig}.
\end{proof}

\begin{algorithm} 
\caption{Pruning algorithm for a simplicial toric ring of dimension $3$.}
\label{alg:pruning}
\begin{flushleft}
    \textbf{Input:} $I \subset R = \kx$ a simplicial toric ideal of dimension $3$ with variables in Noether position \\
    \textbf{Output:} The sets of monomials $\mB_0,\mB_1,\mB_2 \subset R$ that appear in the short resolution of $R/I$.
\end{flushleft}
\begin{algorithmic}[1]
\State $\mG$  $\gets$ reduced Gröbner basis of $I$ for the $\omega$-graded reverse lexicographic order.
\State $\mB_0, \mB_1', \mB_2'$ $\gets$ sets obtained in Algorithm~\ref{alg:dim3}.
\State $C \gets \{v\cdot x_{n-1}^b \in \mB_1' \mid v\in \mB_0 \text{ and } b\geq 2\}$.
\State $r_\alpha \gets$ remainder of $\bx^\alpha$ by $\mG$, $\forall \bx^\alpha \in \mB_1'$.
\State $C_1 \gets \{ \bx^\alpha \in C \mid x_{n-2} \text{ divides } r_\alpha \}$.
\State $C_2 \gets \{ \bx^\alpha \in C \mid x_{n-2} \text{ does not divide } r_\alpha \}$.
\State $\mB_1 \gets \left( \mB_1' \setminus C \right) \cup \{ \bx^\alpha \in C_1 \mid \frac{\bx^\alpha}{x_{n-1}} \notin I + \id{x_{n-2}} \} \cup \{\bx^\alpha \in C_2 \mid \frac{\bx^\alpha}{x_{n-1}} \notin I+\id{x_{n-2}} \text{ or } \frac{r_\alpha}{x_n} \notin I+\id{x_{n-2}}\}$.
\State $\mB_2 \gets \{\bx^\alpha \in \mB_2' \mid \frac{\bx^\alpha}{x_{n-1}} \in I+\id{x_n}\}$.
\end{algorithmic}
\end{algorithm}

Using Theorem~\ref{thm:prunB1} and Theorem~\ref{thm:prunB2}, one can obtain the set $\mB_2 \subset \mB_2'$ in the short resolution. The whole pruning algorithm (for $\mB_1'$ and $\mB_2'$) is summarized in Algorithm~\ref{alg:pruning}. It is worth pointing out that this algorithm requires only the computation the Gro\"bner basis $\mG$ of $I$ with respect to the $\omega$-graded reverse lexicographic order, to compute the remainders of several monomials modulo $\mG$, and to test membership of several monomials to the ideal $I + \id{x_{n-2}}$. Algorithm~\ref{alg:pruning} has been implemented in the function {\tt pruningDim3} of \cite{github_shortres}. \newline

The pruning algorithm presented in this section does not work if the ideal $I$ is not toric. If $I$ is not prime, it can happen that ${\rm pd}_A(R/I) = 3$, so the resolution has one more step and even Algorithm~\ref{alg:dim3} fails. If $I$ is prime but not binomial, Example~\ref{ex:prune_nontoric} shows that Algorithm~\ref{alg:pruning} can fail. 

\begin{example} \label{ex:prune_nontoric}
Set $R := \Q[x_1,\ldots,x_7]$, and let $I \subset R$ be the ideal 
\[I = \id{x_1+t_1^2t_2^2-t_1^3t_3, x_2-t_1^3t_2, x_3-t_2^3t_3, x_4-t_2t_3^3, x_5-t_1^3, x_6-t_2^3, x_7-t_3^3} \cap R \, .\]
The ideal $I$ is prime, $\dim(R/I) = 3$, and the variables are in Noether position. However, $I$ is not binomial, so it is not toric. Applying the results of Section~\ref{sec:groebner}, we obtain $|\mB_0| = 28$, $|\mB_1'| = 16$ and $|\mB_2'| = 4$. Moreover, this resolution is minimal since its Betti diagram, obtained by using the function {\tt shortRes} of \cite{github_shortres}, is 
    \begin{verbatim}
           0     1     2
------------------------
    0:     1     -     -
    1:     4     -     -
    2:     9     2     -
    3:    13    12     3
    4:     1     2     1
------------------------
total:    28    16     4
    \end{verbatim}
However, when applying Algorithm~\ref{alg:pruning} to the sets $\mB_0$, $\mB_1'$ and $\mB_2'$, one gets $\mB_1 = \mB_1'$ and $|\mB_2'\setminus \mB_2| = 1$, so the algorithm fails in this case.
\end{example}

\section{Dependence on the characteristic of $\k$}

In this last section, we present an example of a simplicial toric ring $R/I_\mA$ whose minimal graded free resolution depends on the characteristic of $\k$. Let $\mA \subset \N^6$ be the set defined by the column vectors of the following matrix
\[\left(\begin{array}{rrrrrrrrrrrrrrrr}
3 & 3 & 3 & 3 & 3 & 1 & 1 & 1 & 1 & 1 & 2 & 0 & 0 & 0 & 0 & 0 \\
3 & 3 & 1 & 1 & 1 & 3 & 3 & 3 & 1 & 1 & 0 & 2 & 0 & 0 & 0 & 0 \\
3 & 1 & 3 & 1 & 1 & 3 & 1 & 1 & 3 & 3 & 0 & 0 & 2 & 0 & 0 & 0 \\
1 & 1 & 3 & 3 & 1 & 1 & 3 & 3 & 3 & 1 & 0 & 0 & 0 & 2 & 0 & 0 \\
1 & 1 & 1 & 3 & 3 & 3 & 3 & 1 & 1 & 3 & 0 & 0 & 0 & 0 & 2 & 0 \\
1 & 3 & 1 & 1 & 3 & 1 & 1 & 3 & 3 & 3 & 0 & 0 & 0 & 0 & 0 & 2
\end{array}\right) \, ,\]
and consider the toric ideal $I \subset R = \k[x_1,\ldots,x_{16}]$ determined by $\mA$. Set $A = \k[x_{11},\ldots,x_{16}]$. Then, $I_\mA$ is \whom{} for $\omega = (6,\ldots,6,1,\ldots,1)$ and $A$ is a Noether normalization of $R/I$. \newline

We exhibit here the Betti diagrams of the
short resolution of $R/I$ when $\k$ is a field of characteristic $0$ and when its characteristic is $2$ (computed using the function {\tt shortRes} of \cite{github_shortres}). \newline

\begin{multicols}{2}
\begin{enumerate}
\item[] ${\rm char}(\k) = 0$
    \begin{Verbatim}[fontsize=\small]
           0     1     2
------------------------
    0:     1     -     -
    1:     -     -     -
    2:     -     -     -
    3:     -     -     -
    4:     -     -     -
    5:     -     -     -
    6:    10    15     6
------------------------
total:    11    15     6
    \end{Verbatim}

\columnbreak
\item[] ${\rm char} (\k) = 2$
    \begin{Verbatim}[fontsize=\small]
           0     1     2     3
------------------------------
    0:     1     -     -     -
    1:     -     -     -     -
    2:     -     -     -     -
    3:     -     -     -     -
    4:     -     -     -     -
    5:     -     -     -     -
    6:    10    15     6     1
    7:     -     -     1     -
------------------------------
total:    11    15     7     1
    \end{Verbatim}
\end{enumerate}
\end{multicols}

\medskip

The previous computation shows that the short resolution of a simplicial toric ideal may depend on the characteristic of $\k$ for simplicial toric ideals. Moreover, in our example, the projective dimension as $A$-module is different for both characteristics, ${\rm  pd}_A(R/I) = 2$ when ${\rm char}(\k) = 0$, while ${\rm  pd}_A(R/I) = 3$ when ${\rm char}(\k) = 2$. Since ${\rm pd}_R (R/I) = {\rm pd}_A(R/I) +n-d = {\rm pd}(R/I) +10$, then the resolution of $R/I$ as $R$-module also depends on the characteristic of $\k$. Moreover, the second step of the resolution also depends on the characteristic, as the following direct computation shows.

\newpage

\begin{multicols}{2}
\begin{enumerate}
\item[] ${\rm char}(\k) = 0$
    \begin{Verbatim}[fontsize=\small]
           0     1     2
------------------------
    0:     1     -     -
    1:     -     -     -
    2:     -     -     -
    3:     -     -     -
    4:     -     -     -
    5:     -     -     -
    6:     -    15     6
    7:     -     -     -
    8:     -     -     -
    9:     -     -     -
   10:     -     -     -
   11:     -    55   150
   12:     -     -     -
   13:     -     -     -
   14:     -     -     -
   15:     -     -     -
   16:     -     -   330
------------------------
total:     1    70   486
    \end{Verbatim}

\columnbreak
\item[] ${\rm char}(\k) = 2$
    \begin{Verbatim}[fontsize=\small]
           0     1     2
------------------------
    0:     1     -     -
    1:     -     -     -
    2:     -     -     -
    3:     -     -     -
    4:     -     -     -
    5:     -     -     -
    6:     -    15     6
    7:     -     -     1
    8:     -     -     -
    9:     -     -     -
   10:     -     -     -
   11:     -    55   150
   12:     -     -     -
   13:     -     -     -
   14:     -     -     -
   15:     -     -     -
   16:     -     -   330
------------------------
total:     1    70   487
    \end{Verbatim}
\end{enumerate}
\end{multicols}

\end{document}